\documentclass[11pt]{article}
\usepackage{amsfonts, amsmath, amssymb, amsgen, amsthm, amscd,
latexsym}
\usepackage{color}
\usepackage[all]{xy}

\def\Id{\mathop{\rm Id}\nolimits}
\def\id{\mathop{\rm Id}\nolimits}

\def\Id{\mathop{\rm Id}\nolimits}

\def\Cb{{\mathbb C}}

\def\Bc{{\cal B}}

\def\Gc{{\cal G}}

\def\Ic{{\cal I}}

\def\Kc{{\cal K}}

\def\a{\alpha}
\def\b{\beta}
\def\d{\delta}
\def\D{\Delta}

\def\s{\sigma}

\def\t{\theta}

\def\ve{\varepsilon}

\def\ot{\otimes}

\def\ra{\rightarrow}

\def\lt{\triangleleft}

\def\p{\partial}

\def\0D{\Delta^{(0)}}
\def\1D{\Delta^{(1)}}
\def\Db{\blacktriangledown}
\def\Co{\,\square\,}

\def\td{\tilde}

\newtheorem{theorem}{Theorem}[section]
\newtheorem{remark}[theorem]{Remark}
\newtheorem{proposition}[theorem]{Proposition}
\newtheorem{lemma}[theorem]{Lemma}

\newtheorem{example}[theorem]{Example}
\newtheorem{definition}[theorem]{Definition}

\def\ni{\noindent}

\def\build#1_#2^#3{\mathrel{
\mathop{\kern 0pt#1}\limits_{#2}^{#3}}}
\newcommand{\ns}[1]{~\hspace{-4pt}_{_{{<#1>}}}}
\newcommand{\ps}[1]{~\hspace{-4pt}^{^{(#1)}}}
\newcommand{\sns}[1]{~\hspace{-4pt}_{_{{<\overline{#1}>}}}}
\newcommand{\nsb}[1]{~\hspace{-4pt}_{^{[#1]}}}

\numberwithin{equation}{section}

 \newcommand{\ie}{{\it i.e.\/}\ }

\def\a{\alpha}
\def\b{\beta}

\def\d{\delta}

\def\s{\sigma}
\def\t{\theta}
\def\ve{\varepsilon}

\def\D{\Delta}

\def\ot{\otimes}
\def\part{\partial}

\def\ra{\rightarrow}

\def\text{\hbox}

\def\ot{\otimes}

\def\ra{\rightarrow}

\def\Id{\mathop{\rm Id}\nolimits}

\def\can{\mathop{\rm \bf can}\nolimits}

\def\build#1_#2^#3{\mathrel{
\mathop{\kern 0pt#1}\limits_{#2}^{#3}}}

\numberwithin{equation}{section}
\newcommand{\comment}[1]{\relax}

\begin{document}
\title{\bf  Quantum Groupoids and their Hopf Cyclic Cohomology }

\author{
\begin{tabular}{cc}
Mohammad Hassanzadeh \thanks{Department of Mathematics  and   Statistics,
     University of New Brunswick, Fredericton, NB, Canada}\quad and \quad  Bahram Rangipour$~^\ast$
      \end{tabular}}
\date{}
\maketitle

\begin{abstract}
\ni A new  quantization of groupoids under the name of  {\it $\times$-Hopf coalgebras} is introduced.  We develop   a Hopf cyclic theory with coefficients   in
 stable-anti-Yetter-Drinfeld modules for $\times$-Hopf coalgebras. We use $\times$-Hopf coalgebras to study coextensions of coalgebras.   Finally,  equivariant $\times$-Hopf coalgebra Galois coextensions are defined and applied  as functors between categories of stable anti-Yetter-Drinfeld modules over $\times$-Hopf coalgebras involved in the coextension.
\end{abstract}

\section*{Introduction}

 It is  known that $\times$-Hopf algebras quantize groupoids \cite{b1}. However,  this quantization  is not comprehensive  enough \cite{bs2}. To wit, let us  consider the  natural  cyclic module, \`ala Connes \cite{NCG}, defined as the  free module  generated by the nerve of a groupoid; it is then  easy to see that the  Hopf cyclic complex of the groupoid algebra viewed as a $\times$-Hopf algebra  does  not coincide with the  cyclic complex of the cyclic module.

To remedy  this lack of consistency,  we structure  our correct gadget  on the groupoid coalgebra of a groupoid instead of groupoid algebra. This way we obtain an object that  is called  $\times$-Hopf coalgebra.  Naturally  $\times$-Hopf coalgebras  generalize  Hopf algebras and week Hopf algebras. The simplest example of such an object is the enveloping coalgebra of  a coalgebra.

After  defining  $\times$-Hopf coalgebras axiomatically, we study their  representation and corepresentation to define   stable-anti-Yetter-Drinfeld modules over   $\times$-Hopf coalgebras. Next, we introduce a  Hopf cyclic theory for $\times$-Hopf coalgebras with coefficients in stable-anti-Yetter-Drinfeld modules.   We observe that the cyclic complex of groupoid coalgebra perfectly   coincides with the nerve generated cyclic complex.

The use of Hopf algebras in  study of   (co)algebra (co)extensions is not new, for example see \cite{bh}.   $\times$-Hopf algebras are  used to define appropriate Galois extension of algebras over algebras \cite{bs2}. Similarly we use  $\times$-Hopf coalgebras  to define Galois coextensions of coalgebras over coalgebras.
By the work of Jara-Stefan \cite{JaSt:CycHom} and B\"ohm-Stefan \cite{bs2} one knows that Hopf Galois extensions are an interesting  source of stable-anti-Yetter-Drinfeld modules. This construction was generalized by our previous work \cite{hr}: any equivariant Hopf Galois extension defines a functor between the category of stable-anti-Yetter-Drinfeld modules of the $\times$-Hopf algebras involved in the extension.  To state the corresponding  result,  we introduce equivariant $\times$-Hopf
coalgebra Galois coextensions and use them as a functor between the category of stable anti Yetter-Drinfeld modules of the $\times$-Hopf coalgebras associated with  the coextension.

\vspace{1cm}

\textbf{Notations}:
In this paper all objects are vector spaces over  $\Cb$,  the  field of complex numbers. Coalgebras are denoted by $C$ and  $D$, and Hopf algebras are denoted by  $H$.   We use  $\mathcal{K}$ and $\mathcal{B}$ as  left and right $\times$-Hopf coalgebra, respectively. We denote left coactions  of coalgebras by $\Db(a)=a\sns{-1}\ot a\sns{0}$,  left  coactions of right $\times$-Hopf coalgebras by  $\Db(a)=a\nsb{-1}\ot a\nsb{0}$,  and left coaction of left $\times$-Hopf coalgebra  by $\Db(a)=a\ns{-1}\ot a\ns{0}$. We use similar notations for  right coactions of the above objects. If $X$ and $Y$ are right and left  $C$-comodules, we define their cotensor product $X\Co_C Y$ as follows,
\begin{equation*}
  X\Co_C Y:= \left\{x\ot y; \quad x\sns{0}\ot x\sns{1}\ot y= x\ot y\sns{-1}\ot y\sns{0}\right\}
\end{equation*}


\tableofcontents


\section{Motivation: Groupoid homology revisited}

\label{section-1}

In this section we briefly recall groupoids and their homology. We justify the need of    $\times$-Hopf coalgebras as a   reasonable, and missed,   quantization of groupoids.

It is known that the  Hopf cyclic homology, dual theory,  of Hopf algebras generalizes group homology \cite{bm1}.  One associates a cyclic module  to any Hopf algebra equipped with an extra structure such as a modular pair in involution,  or  a stable anti Yetter-Drinfeld module in general, \cite{ConMos:HopfCyc}(see also\cite{bm1, HaKhRaSo1,HaKhRaSo2}). When the Hopf algebra is a group algebra  one observes that the corresponding cyclic module coincides with the cyclic  module associated to the group in question \cite{bm1}.

\medskip

We would like to define appropriate  Hopf algebraic structure and define its cyclic theory  to  generalizes groupoids and their  cyclic theory.   One knows that  the groupoid algebra of a groupoid is  a $\times$-Hopf algebra. However its cyclic complex  does not coincide with the cyclic complex  of the groupoid  generated by the nerve  \cite{bs2}.

We denote a  groupoid  by $\Gc=(\Gc^1, \Gc^0)$. Here $\Gc^1$ denotes the set of morphisms and $\Gc^0$ is the set of objects. We denote the source and target maps by $t,s: \Gc^1\ra \Gc^0$, \ie, for  $g\in {\rm Mor}(x,y)$, $s(g)=x$, and $t(g)=y$.

A groupoid is called cyclic if for any $A\in \Gc^0$ there is a morphism $\t_A\in {\rm Mor}(A,A)$ such that $f\t_A=\t_Bf$ for any $f\in {\rm Mor}(A,B)$ . Any groupoid is trivially  cyclic via $\t^A=\Id_A$. We denote a groupid $\Gc$ equipped with a cyclic structure $\t$ by $(\Gc,\t)$.
 Let us recall the nerve of a groupoid from  \cite{bur}. For a groupoid $\Gc=(\Gc^1,\Gc^0)$  we denote by $\Gc_n$, for $n\ge 2$,  the nerve of $\Gc$, as  the set of all $(g_1, g_2, \ldots,g_n)$ such that
 $s(g_i)=t(g_{i+1})$, and $\Gc_i=\Gc^i$, for $i=0,1$.  Then one easily observes that for any cyclic groupoid $(\Gc, \t)$,  $\Gc_{\bullet}$ becomes a cyclic set with the following morphisms,

 \begin{align}
&\p_i:\Gc_{\bullet}\ra \Gc_{\bullet-1}, \qquad 0\le i\le \bullet\\
&\s_j: \Gc_\bullet\ra \Gc_{\bullet+1}, \qquad 0\le j\le \bullet\\
&\tau:\Gc_{\bullet}\ra \Gc_{\bullet}
 \end{align}

 which are defined by
 \begin{align}
& \p_0[g_1,\ldots,g_n]= [g_2,\ldots,g_n],\\
 &\p_i[g_1,\ldots,g_n]= [g_1,\ldots,g_ig_{i+1},\ldots, g_n],\\
 &\p_n[g_1,\ldots,g_n]= [g_1,\ldots,g_{n-1}],\\
 &\s_j[g_1,\ldots,g_n]= [g_1,\ldots, g_{i-1}, \Id_{s(g_i)},g_i \ldots, g_n ],\\
 &\tau[g_1,\ldots,g_n]= [\t_{s(g_n)}g_n^{-1}\cdots g_1^{-1}, g_1, \ldots, g_{n-1}].
 \end{align}

 The special cases $\p_0,\p_1:\Gc_1\ra \Gc_0$ are defined by $\p_0=s$, and $\p_1=t$.  Finally $\s_0:\Gc_0\ra \Gc_1$  is defined by $\s_0(A)=\Id_A$.  We denote this cyclic module by $C_\bullet(\Gc,\t)$ and its cyclic homology  by $HC_\bullet(\Gc,\t).$

Now let $\Bc=\Cb\Gc^1$ be the coalgebra generated by $\Gc^1$, \ie as a module it is the  free module generated by $\Gc$ with $\D(x)=x\ot x$, $\ve(x)=1$.  Similarly  we set $C:=\Cb\Gc^0$  as the coalgebra generated by $\Gc^0$.

The source and target maps induce the following morphisms of coalgebras
\begin{equation}\label{last1}
\a, \b:\Bc\ra C,  \quad \a(g):=s(g), \quad \b(g):=t(g).
\end{equation}
The maps $\a$ and $\b$ induce a  $C$-bicomodule structure on $C$ as follows,
\begin{align}
&\Db_L:\Bc\ra C\ot \Bc, \qquad\Db_L(g)= \b(g)\ot g\\
&\Db_R:\Bc\ra \Bc\ot C, \qquad \Db_R(g)=g\ot \a(g).\label{last2}
\end{align}

 One then identifies
 $\Bc\Co_C \Bc$ with $\Gc_2(\Gc)$ in the usual way.
 Here the cotensor space is defined by
 \begin{equation}\Bc\Co_C\Bc :=\left\{g\ot h\in \Bc\ot \Bc\mid g\ps{1}\ot \a(g\ps{2})\ot h=g\ot \b(h\ps{1})\ot h\ps{2}\right\}
 \end{equation}

Next, one  uses the composition rule of $\Gc$ to define a multiplication as follows,
  \begin{equation}
 \mu: \Bc\Co_C \Bc\ra \Bc, \quad \mu(g\ot h)=gh.
 \end{equation}

 There is also a map $\eta: C\ra \Bc$ defined by $\eta(A)=\Id_A$. The inverse of $\Gc$ induces the following map
 \begin{equation}\label{last3}
\nu^{-1}: \Bc\Co_{C_{\rm cop}} \Bc\ra \Bc\Co_C\Bc, \quad \nu^{-1}(g\ot h)= g\ot g^{-1}h,
 \end{equation}
 which defines an inverse for the map
 \begin{equation}
 \nu: \Bc\Co_C \Bc\ra \Bc\Co_{C_{\rm cop}}\Bc, \quad \nu(g\ot h)= g\ot gh.
 \end{equation}

In the next step  we define our coefficients  $C$ to be acted and coacted by $\Bc$. First we define the coaction of $\Bc$ on $C$ by
 \begin{equation}\label{coact-group}
 \Db:C\ra   \Bc\ot C, \quad \Db(A)= \t_A\ot A.
 \end{equation}

 We also let $\Bc$ act on $C$ via
 \begin{align}\label{act-group}
 C\Co_C \Bc\ra C, \qquad c\cdot g=\ve(c)s(g)
 \end{align}
 By this information we offer  a new cyclic module;
 \begin{equation}
 C_\bullet(\Bc,~ ^\t C):= \underset{\bullet+1~~{\rm  times}}{ \underbrace{\Bc \Co_C\Co\cdots\Co_C \Bc} }\Co_\Bc C.
 \end{equation}

  Here $C^n(\Bc,~^\t C)$ can be identified with the vector space generated by  all elements $[ g_0,\ldots, g_n,c]\in  \Gc^1\times\cdots\times\Gc^1\times \Gc^0 $ such that
  \begin{equation}
  s(g_i)=t(g_{i+1}), s(g_n)=t(g_0) \quad \text{and }\quad g_0\cdots g_n=\t_c.
   \end{equation}
   One easily observe that $C_\bullet(\Bc,~^\t C)$ is a cyclic module via the following morphisms.
    \begin{align}
    &\p_i([g_0, \ldots,g_n,c])=[ g_0, \ldots, g_ig_{i+1}, \ldots, g_n,c], \\
    &\p_n([g_0, \ldots,g_n,c])=[g_ng_0, g_1, \ldots, g_{n-1}, c\cdot  g_n^{-1}],\\
    &\s_j([g_0, \ldots,g_n,c])= [g_0, \ldots,g_i,\id_{t(g_i)},g_{i+1}, \ldots, g_n,c],\\
    &\tau([g_0, \ldots,g_n,c])= [ g_n,g_0, \ldots, g_{n-1}, c\cdot g_n^{-1}].
    \end{align}

   As the final step one observes that $C_\bullet(\Bc,~^\t C)$ and $C_\bullet(\Gc,\t)$  are isomorphic as cyclic modules via
   \begin{align}
  & \Ic: C_\bullet(\Bc,~^\t C)\ra C_\bullet(\Gc,\t),\quad \Ic([g_0, \ldots, g_n,c])= [g_1,g_2, \ldots,g_n],
   \end{align}
   with the inverse map given by $\Ic^{-1}:  C_\bullet(\Gc,\t)\ra C_\bullet(\Bc,~^\t C),$
   \begin{equation}
   \Ic^{-1}([g_1, \ldots, g_n])= [ \t_{s(g_n)}(g_1\ldots g_n)^{-1}, g_1, \ldots, g_n, s(g_n)].
   \end{equation}

As we see above, we are faced with a new structure which is not  $\times$-Hopf algebra. Our aim in the sequel sections is to study such objects and define a cyclic theory for them (see Example \ref{groupoid1}). Also we observe   that  the desired  cyclic theory is capable to handle coefficients. The above cyclic structure $\t$ defines an example of such  coefficients (see Remark \ref{groupoid2}).


\section{ $\times$-Hopf coalgebras}\label{SS-2-1}
In this section we introduce $\times$-Hopf coalgebras as a formal dual of $\times$-Hopf algebras. We carefully
 show that they define symmetries  for (co)algebras with several objects. We apply them to define Hopf cyclic cohomology of (co)algebras with several objects under the symmetry of $\times$-Hopf coalgebras.

\subsection{Bicoalgebroids and $\times$-Hopf coalgebras}
It is well known that the category of $\times$- Hopf algebras
contains Hopf algebras as a full subcategory. However,
$\times$-Hopf algebras are not a  self-dual  alike  Hopf algebras.
In fact this lack of self-duality is  passed on to the $\times$-Hopf
algebras from bialgebroids.

 There  are many generalizations of Hopf algebras in
 the literature and  the most natural ones are  $\times$-Hopf algebras \cite{sch}.
  In this paper we  are interested in the formal dual  of $\times$-Hopf algebras.
  For the convenience of the reader  we  present a very short description of $\times$-Hopf algebras.
  Let $R$ and $\Kc$ be algebras
  with two algebra maps $\mathfrak{s}:R\longrightarrow \Kc$
   and $\mathfrak{t}:R^{op}\longrightarrow \Kc$, called source and target, such that their
  ranges commute with one another. We equip  $\Kc$  and $\Kc\ot_R \Kc$
  with  $R$-bimodule structures using the source and target maps. More precisely, $r\triangleright k=\mathfrak{s}(r)k$ and $k\triangleleft r =  \mathfrak{t}(r)k$.
   We assume that there are $R$-bimodule  maps  $\Delta:\Kc\longrightarrow \Kc\otimes_{R} \Kc$ and
     $\varepsilon:\Kc\longrightarrow R$ via which   $\Kc$ is an $R$-coring.
     The data $(\Kc,\mathfrak{s},\mathfrak{t},\D,\ve)$ is called a left $R$-bialgebroid if
      for $k_1,k_2\in \Kc$ and $r\in R$ the following identities hold

\begin{enumerate}
 \item [i)] $k^{(1)} \mathfrak{t}(r)\otimes_{R} k^{(2)} = k^{(1)} \otimes_{R} k^{(2)} \mathfrak{s}(r),$
\item[ii)] $\Delta(1_{\Kc})=1_{\Kc}\ot_R 1_{\Kc}, \quad \text{ and} \quad  \Delta(k_1k_2)=k_1^{(1)} k_2^{(1)}\otimes_R k_1^{(2)} k_2^{(2)}$.
\item[iii)] $\varepsilon(1_{\Kc})=1_{R} \quad \text{and} \quad \varepsilon(k_1k_2)=\varepsilon(k_1\mathfrak{s}(\varepsilon(k_2))).$
\end{enumerate}
 A left $R$-bialgebroid $(\Kc,\mathfrak{s},\mathfrak{t},\Delta,\ve)$ is said to be a left $\times_{R}$-Hopf algebra if the  Galois map
\begin{align}\label{aaaa}
\nu: \Kc\otimes_{R^{op}}\Kc\longrightarrow \Kc\otimes_{R}\Kc, \qquad  k\otimes_{R^{op}}k'\longmapsto k^{(1)}\otimes_{R}k^{(2)}k',
\end{align}
is bijective.  The role of antipode in $\times_R$-Hopf algebras is played by the following map,
\begin{align}
\begin{split}
\Kc\ra \Kc\ot_{R^{op}}\Kc, \quad k\mapsto \nu^{-1}(k\ot_{R} 1_{\mathcal{K}}).
\end{split}
\end{align}
Similarly, one  has the definition of right bialgebroids and right $\times$-Hopf algebras.

\bigskip

 Let us  define our main object of this paper. To do so,  we first  recall
  bicoalgebroid from  \cite{bm}. Roughly speaking,  bicoalgebroids
 are formal  dual  of  bialgebroids as they are defined  by reversing the arrows in the definition of bialgebroids.
Let  $ \Kc$ and  $C$ be two coalgebras with  coalgebra maps $\alpha: \Kc \to C$ and $\beta: \Kc \to C_{{\rm cop}}$,
such that their images  cocommute, \ie
\begin{equation}
\alpha (h^{(1)}) \ot \beta(h^{(2)}) = \alpha (h^{(2)}) \ot \beta(h^{(1)}).
\end{equation}

 These maps endow  $\Kc$ with a $C$-bicomodule
structure, via the left and right  $C$-coactions as follows,
\begin{equation}
 \blacktriangledown_L(h)= \alpha(h\ps{1})\ot h\ps{2}, \quad \blacktriangledown_R(h)= h\ps{2}\ot \beta(h\ps{1}).
\end{equation}
The next we assume  two $C$-bicomodule maps  $\mu_\Kc: \Kc \square_C \Kc \longrightarrow \Kc$ and $\eta_\Kc: C  \longrightarrow \Kc$ called the multiplication and the unit,   respectively, which  satisfy  the following axioms:
\smallskip

\begin{align}\label{tak}
& i) \qquad \sum_i \mu(g_i \ot h_i^{(1)}) \ot \alpha(h_i^{(2)}) = \mu(g_i^{(1)}\ot h_i)\ot \beta(g_i^{(2)}),\\
\label{brzmil2}
& ii)\qquad  \Delta \circ \mu (\sum_i g_i \ot h_i) = \sum_i \mu (g_i^{(1)} \ot
h_i^{(1)}) \ot \mu (g_i^{(2)} \ot h_i^{(2)}),\\
\label{var}
& iii)\qquad \varepsilon (g) \varepsilon (h) = \varepsilon \circ \mu (g \ot h),\\
\label{2.8}
& iv)\qquad  \mu \circ (\eta \Co \Id_\Kc) \circ \Db_L = \Id_\Kc= \mu \circ (\Id_\Kc \Co \eta)
\circ \Db_R,\\
& v)\qquad \label{etaalphaeta} \Delta (\eta (c)) = \eta (c)^{(1)} \ot \eta
(\alpha (\eta(c)^{(2)})) = \eta (c)^{(1)} \ot \eta (\beta
(\eta(c)^{(2)})),\\
&vi)\qquad \label{veeta} \varepsilon (\eta (c)) = \varepsilon (c).
\end{align}

Furthermore the multiplication map $\mu$ is associative. We call  $( \Kc, \Delta_\Kc, \varepsilon_\Kc, \mu, \eta, \alpha, \beta, C )$ a left bicoalgebroid \cite{bm} (also see\cite{bal}). For simplicity,  we use the notation $\mu(h\ot k)= hk$ for all $h,k\in \Kc$.
It is shown in  \cite{bm} and \cite{bal} that  the conditions i), \dots, vi) make sense.

\begin{lemma}\label{leftbi}
 Let  $\mathcal{K}=( \Kc, C )$ be a left bicoalgebroid. Then the following properties hold for all $h,k\in \mathcal{K}$ and $c\in C$.
\begin{itemize}
 \item[i)] $\alpha(\eta(c))=c.$
\item[ ii)]$ \beta(\eta(c))=c.$
\item[ iii)]$\alpha(hk)=\alpha(h)\ve(k).$
\item[ iv)]$\beta(hk)= \ve(h)\beta(k).$
\end{itemize}
\end{lemma}
\begin{proof}
    The relations i) and ii) are obtained by applying $\Id  \ot \ve$ and $\ve \ot \Id$ on the both hand sides of
    \begin{align*}
      & c\ps{1}\ot \eta(c\ps{2})= \alpha(\eta(c)\ps{1})\ot \eta(c)\ps{2},\\
      &\eta(c\ps{1}) \ot c\ps{2}= \eta(c)\ps{2}\ot \beta(\eta(c)\ps{1}),
    \end{align*}
    which are left and right $C$-comodule map properties of the unit map $\eta$, respectively.
     The relations iii) and iv) are proved in \cite[page 8]{bal}.
\end{proof}
 One notes  that the relations i) and ii) in the previous  lemma are dual to the relations $\ve(t(r))=\ve(s(r))=r$  and also the relations iii) and iv) are dual to the relations  $\Delta(s(r))=s(r)\ot 1$ and $\Delta(t(r))= 1\ot t(r)$ for left $\times_R$-bialgebroids.

\begin{definition}
A left bicoalgebroid $\Kc$ over  the coalgebra $C$ is said to be a left $\times_C$-Hopf coalgebra if the following Galois map
\begin{align}\label{nu}
\nu: \mathcal{K} \Co_C \mathcal{K}\longrightarrow \mathcal{K}\Co_{C_{{\rm cop}}} \mathcal{K}, \quad k \ot k'\longmapsto k k'\ps{1}\ot k'\ps{2},
\end{align}
is bijective. Here in the codomain of $\nu$, the $C_{{\rm cop}}$-comodule structures are given by right and left coactions defined by $\beta$,  \ie

\begin{equation}
k\longmapsto \beta(k\ps{1})\ot k\ps{2}, \quad  k\longmapsto k^{(1)}\ot \beta(k^{(2)}).
\end{equation}

In the domain of $\nu$, the $C$-comodule structures are given by the original  right and left coactions  defined by $\alpha$ and $\beta$, \ie
\begin{equation}
k\longmapsto \alpha(k\ps{1})\ot k\ps{2}, \quad k\longmapsto k\ps{2} \ot \beta(k\ps{1}).
\end{equation}
\end{definition}
To show  that $\nu$ is well-defined,  we   prove that
$k k'\ps{1}\ot k'\ps{2}\in K\Co_{C_{{\rm cop}}} K.$  Indeed,
\begin{align*}
&(k k'^{(1)})^{(1)} \ot \beta((k k'^{(1)})^{(2)}) \ot k'^{(2)}=k^{(1)} k'^{(1)} \ot \beta(k^{(2)}\ot k'^{(2)}) \ot k'^{(3)}\\
&=k^{(1)}\varepsilon(k^{(2)}) k'^{(1)}\ot \beta(k'^{(2)})\ot k'^{(3)}=k k'^{(1)}\ot \beta(k'^{(2)(1)})\ot k'^{(2)(2)}.
\end{align*}

Here  Lemma \eqref{leftbi}{(iv)}  is used in the second equality.

 We use the following summation notation  for the image of $\nu^{-1}$,
\begin{equation}
\nu^{-1}(k\ot k')= \nu^{-}(k,k')\ot \nu^{+}(k, k').
\end{equation}
When there is no confusion we set $\nu^{-1}=\nu^{-}\ot \nu^{+}$.

 \begin{lemma}\label{property} Let $C$ be a coalgebra and $\mathcal{K}$ a left $\times_C$-Hopf coalgebra. Then the following properties hold for $\nu$ and $\nu^{-1}$.\\
\begin{enumerate}
\item[i)] $ \nu^{-} \nu^+\ps{1}\ot \nu^+\ps{2}= \Id_{\mathcal{K}\Co_{C_{{\rm cop}}} \mathcal{K}}.$
\item[ii)]  $\nu^{-}(k k'\ps{1}, k'\ps{2}) \ot \nu^{+}(k k'\ps{1}, k'\ps{2})= k\ot k'.$
 \item[iii)]  $\ve(\nu^{-}(k\ot k')) \ve(\nu^{+}(k\ot k'))= \ve(k)\ve(k').$
 \item[iv)]  $\ve(\nu^-(k\ot k'))\nu^+(k\ot k')=\ve(k)k'.$
\item[v)]  $\nu^-(k\ot k') \nu^+(k\ot k')= \ve(k')k.$
\end{enumerate}

 \end{lemma}
 \begin{proof}
   The relation i)  is equivalent to $\nu \nu^{-1}= \Id$. The relation ii) is equivalent to $\nu^{-1}\nu= \Id$. The relation iii) is proved by applying $\ve \ot \ve$ to the both sides of i) followed by \eqref{var}. The relation iv)  is obtained by applying $\ve\ot \Id$ on both hand sides of i). The relation v)  is shown  by applying $\Id\ot \ve$ on both hand sides of i).
 \end{proof}

One notes that for a left $\times_C$-Hopf coalgebra $\mathcal{K}$, the maps $\nu$ and hence $\nu^{-1}$ are both right $\mathcal{K}$-comodule maps where the right $\mathcal{K}$-comodule structures of $\mathcal{K}\Co_C \mathcal{K}$ and $\mathcal{K}\Co_{C_{\rm cop}} \mathcal{K}$ are given by $k\ot k'\longmapsto k\ot k'\ps{1}\ot k'\ps{2}$. The right $\mathcal{K}$-comodule map property of $\nu^{-1}$ is equivalent to
\begin{equation*}
  \nu_-(k\ot k')\ot \nu_+(k\ot k')\ps{1}\ot \nu_+(k\ot k')\ps{2}= \nu_-(k\ot k'\ps{1})\ot \nu_+(k\ot k'\ps{1})\ot k'\ps{2}.
\end{equation*}


\begin{definition}
   A right bicoalgebroid $(\mathcal{B}, \D, \ve,  \mu, \eta, \alpha, \beta, C )$
consists of coalgebras $\mathcal{B}$ and $C$
 with  coalgebra maps $\alpha: \mathcal{B} \to C$ and $\beta: \mathcal{B} \to C_{{\rm cop}}$,
such that their images  cocommute, \ie
$\alpha (b^{(1)}) \ot \beta(b^{(2)}) = \alpha (b^{(2)}) \ot \beta(b^{(1)})$. These maps furnish $\mathcal{B}$ with a $C$-bicomodule
structure, via left and right  $C$-coactions
\begin{equation}\label{rightc}
 \blacktriangledown_{L}(b)= \beta(b\ps{2})\ot b\ps{1}, \quad \blacktriangledown_{R}(b)= b\ps{1}\ot \alpha(b\ps{2})
\end{equation}
 $C$-bicomodule maps $\mu_{\mathcal{B}}: \mathcal{B} \square_C \mathcal{B} \longrightarrow \mathcal{B}$ and $\eta_{\mathcal{B}}: C  \longrightarrow \mathcal{B}$ making $\mathcal{B}$ an algebra in $  ^{C}M^C$, satisfying the following properties:
\smallskip
\begin{align}
\label{tak2}
& i)\quad \sum_i \mu (b_i \ot b'_{i}\ps{2}) \ot \beta(b'_{i}\ps{1}) = \mu (b_i\ps{2} \ot b'_i) \ot \alpha(b_i\ps{1})\\
\label{brzmil23}
& ii) \quad\Delta \circ \mu (\sum_i b_i \ot b'_i) = \sum_i \mu (b_i\ps{1} \ot
b'_i\ps{1}) \ot \mu (b_i\ps{2} \ot b'_i\ps{2})\\
\label{var2}
& iii) \quad\varepsilon (b) \varepsilon (b') = \varepsilon \circ \mu (b \ot b')\\
& iv)\quad \mu \circ (Id \Co \eta) \circ  ^R\blacktriangledown_{C}  = \mathcal{B} = \mu \circ ( \eta \Co Id)
\circ  ^L\blacktriangledown_{C}\\
\label{etaalphaeta2}
&v)\quad \Delta (\eta (c)) = \eta (\alpha
(\eta(c)^{(1)})) \ot \eta(c)\ps{2}  =\eta (\beta
(\eta(c)^{(1)})) \ot \eta(c)\ps{2}\\
\label{epsil}
& vi) \quad \varepsilon (\eta (c)) = \varepsilon (c).
\end{align}
\end{definition}
For simplicity we use the notation $\mu(b\ot b')= bb'$ for all $b, b' \in \mathcal{B}$. One notes that the relation \eqref{tak2} makes sense because if $b\ot b'\in \mathcal{B}\Co_C \mathcal{B}$, then
\begin{equation}\label{nnn1}
b\ot \beta (b'\ps{2})\ot b'\ps{1}= b\ps{1} \ot \alpha(b\ps{2})\ot b'.
\end{equation}
By applying $\Id \ot \Id \ot \ve$ on the both sides of \eqref{nnn1} we get
\begin{equation}\label{nnn2}
b\ot \beta (b'\ps{3})\ot b'\ps{1} \ot b'\ps{2}= b\ps{1} \ot \alpha(b\ps{2})\ot b'\ps{1}\ot b'\ps{2},
\end{equation}
 which is equivalent to  $b\ot b'\ps{2}\ot b'\ps{1}\in \mathcal{B}\Co_C \mathcal{B} \ot \mathcal{B}.$ So the left hand side of \eqref{tak2} is well-defined. To show that the right hand side of \eqref{tak2} is well-defined, one applies $\Delta\ot \Id \ot \Id$ on  the both sides of \eqref{nnn1} to obtain
$$b\ps{1}\ot b\ps{2}\ot \beta(b'\ps{2})\ot b'\ps{1}=b\ps{1}\ot b\ps{2} \ot \alpha(b\ps{3})\ot b'.$$

One also notes that the  unit map $\eta: C\longrightarrow \mathcal{B}$ is a right $C$-comodule map, \ie
 \begin{equation}\label{bos}
   \eta(c\ps{1})\ot c\ps{2}= \eta(c)\ps{1}\ot \alpha(\eta(c)\ps{2}),
 \end{equation}
 where $C$ is a right $C$-comodule by $\Delta_C$ and the right $C$-comodule structure of $\mathcal{B}$ is given in \eqref{rightc}.

\begin{lemma}\label{rightbi}
 Let  $C$ be a coalgebra and $ \mathcal{B}$ be a right $\times_C$-bicoalgebroid. The following properties hold for all $b,b'\in \mathcal{B}$ and $c\in C $.
\begin{enumerate}
\item[i)] $\alpha(\eta(c))=c.$
\item[ ii)] $ \beta(\eta(c))=c.$
 \item[iii)] $\alpha(b b'))=\ve(b)\alpha(b'). $
\item[ iv)] $\beta(b b')= \beta(b)\ve(b').$
\end{enumerate}
\end{lemma}
\begin{definition}
Let $C$ be a coalgebra. A right $\times_C$-bicoalgebroid $\Bc$  is said to be a right $\times_C$-Hopf coalgebra provided that the map
\begin{align}\label{nu2}
\nu: \mathcal{B} \Co_C \mathcal{B}\longrightarrow \mathcal{B}\Co_{C_{{\rm cop}}} \mathcal{B}, \quad b \ot b'\longmapsto b\ps{1}\ot b\ps{2} b',
\end{align}
is bijective. In the codomain of this map in \eqref{nu2}, $C_{{\rm cop}}$-comodule structures are given by the right and left coaction via $\beta$, \ie

\begin{equation}
b\longmapsto \beta(b\ps{1})\ot b\ps{2}, \quad  b\longmapsto b^{(1)}\ot \beta(b^{(2)}).
\end{equation}

In the domain of the map in \eqref{nu2}, $C$-comodule structures are given by the original  right and left coactions by $\alpha$ and $\beta$, \ie
\begin{equation}
b\longmapsto \beta(b\ps{2})\ot b\ps{1}, \quad b\longmapsto b\ps{1} \ot \alpha(b\ps{2}).
\end{equation}
\end{definition}
The map $\nu$ introduced in \eqref{nu2} is well-defined  by \eqref{nnn2} and the fact that one has
$$b\ps{1}\ot b\ps{2} b' \in \mathcal{B}\Co_{C_{{\rm cop}}} \mathcal{B}.$$ In fact,
\begin{align*}
&b\ps{1}\ot \beta(b\ps{2}\ps{1} b'\ps{1})\ot b\ps{2}\ps{2} b'\ps{2}=b\ps{1}\ot \beta(b\ps{2} b'\ps{1})\ot b\ps{3}  b'\ps{2}=\\
&b\ps{1}\ot \beta(b\ps{2})\ot b\ps{3}  \ve(b'\ps{1})b'\ps{2}=b\ps{1}\ot \beta(b\ps{2})\ot b\ps{3}  b'.\\
\end{align*}

We use Lemma \eqref{rightbi}{(iv)} in the second equality. We denote the inverse map $\nu^{-1}$ by the following summation notation,
\begin{equation}
\nu^{-1}(b\Co_{C_{{\rm cop}}} b')= \nu^{-}(b,b')\ot \nu^{+}(b, b').
\end{equation}
 If there is no confusion we use $\nu^{-1}:=\nu^{-}\ot \nu^{+}$. Similar to  Lemma \eqref{property}, one can prove the following lemma.

 \begin{lemma}\label{property2} Let $C$ be a coalgebra and $\mathcal{B}$ be a right $\times_C$-Hopf coalgebra. Then the  following properties hold.
\begin{itemize}
   \item[i)] $\nu^-\ps{1}\ot \nu^-\ps{2} \nu^+= \Id_{\mathcal{B}\Co_{C_{{\rm cop}}}\mathcal{B}}.$

   \item[ii)] $ \nu^-(b\ps{1}, b\ps{2} b')\ot \nu^+(b\ps{1}, b\ps{2} b')= b\ot b'.$

  \item[ iii)] $\ve(\nu^{-}(b\ot b')) \ve(\nu^{+}(b\ot b'))= \ve(b)\ve(b').$

  \item[ iv)] $ \ve(\nu^+(b\ot b'))\nu^-(b\ot b')=\ve(b')b.$

   \item[v)] $ \nu^-(b\ot b') \nu^+(b\ot b')= \ve(b)b'.$
\end{itemize}
 \end{lemma}
 One notes that for any right $\times$-Hopf coalgebra $\mathcal{B}$, the map $\nu$ and therefore $\nu^{-1}$ are both left $\mathcal{B}$-comodule maps where the left $\mathcal{B}$-comodule structures of $\mathcal{B}\Co_C \mathcal{B}$ and $\mathcal{B}\Co_{C_{\rm cop}} \mathcal{B}$ are given by $b\ot b'\longmapsto b\ps{1}\ot b\ps{2}\ot b'$. The left $\mathcal{B}$-comodule property  of $\nu^{-1}$ is equivalent to
 \begin{equation}\label{b-comodule-nu}
   \nu_-(b\ot b')\ps{1}\ot \nu_-(b\ot b')\ps{2} \ot \nu_+(b\ot b')= b\ps{1}\ot \nu_-(b\ps{2}\ot b')\ot \nu_+(b\ps{2}\ot b').
 \end{equation}

 \subsection{Examples}
 \begin{example}\label{example-a}{\rm
   The $\times$-Hopf coalgebras generalize  Hopf algebras. We recall that a left Hopf algebra  is a bialgebra $B$ endowed with a left antipode $S$, \ie $S:B\longrightarrow B$ where $S(b\ps{1})b\ps{2}=b$, for all $b\in B$ \cite{lefthopf}. Similarly a right Hopf algebra is a bialgebra endowed with a right antipode. Obviously any Hopf algebra is both a left and a right Hopf algebra. If $H$ is a bialgebra,  it is a left $\times_{\mathbb{C}}$-Hopf coalgebra, if and only if
   \begin{equation}
     \nu(h\ot h')= hh'\ps{1}\ot h'\ps{2}, \qquad \nu^{-1}(h\ot h')= hS(h'\ps{1})\ot h'\ps{2}.
   \end{equation}
  The relation $ \nu^{-1}\nu=id$ implies that $H$ is a  right Hopf algebra and $\nu \nu^{-1}=id$ is equivalent to the fact that $H$ is a  left Hopf algebra. Therefore $H$ should be a Hopf algebra.   Also $H$ is a right $\times_{\mathbb{C}}$-Hopf coalgebra if and only if
  \begin{equation}
    \nu(b\ot b')= b\ps{1}\ot b\ps{2}b', \qquad \nu^{-1}(b\ot b')= b\ps{1}\ot S(b\ps{2})b'.
  \end{equation}}

 \end{example}
 \begin{example}\label{example}{\rm
 The simplest example of a $\times$-Hopf coalgebra which is not a Hopf algebra comes as follows. The co-enveloping coalgebra $C^e= C\ot C_{\rm cop}$ is a left $\times_C$-Hopf coalgebra where the source and target maps are given by
$$\alpha: C\ot C_{{\rm cop}}\longrightarrow C, \quad c\ot c'\longmapsto c\ve (c'); \quad \beta: C\ot C_{\rm {\rm cop}}\longrightarrow C_{{\rm cop}}, \quad c\ot c'\longmapsto \ve(c) c',$$ multiplication by
$$C\ot C_{{\rm cop}}\Co C\ot C_{{\rm cop}}\longrightarrow C\ot C_{{\rm cop}}, \quad c\ot c'\Co d\ot d'\longmapsto \ve(c') \ve(d) c \ot d' ,$$  unit map by
$$\eta: C\longrightarrow C\ot C_{{\rm cop}}, \quad c\longmapsto c\ps{1} \ot c\ps{2},$$ and
$$\nu(c\ot c' \Co_C d\ot d')=\ve(c')  c\ot d'\ps{2}\Co{ C_{{\rm cop}}} d\ot d'\ps{1},$$
$$\nu^{-1}(c\ot c' \Co_{C_{{\rm cop}}}d\ot d')= \ve(c')c\ot d\ps{1}\Co_C d\ps{2}\ot d'.$$}
\end{example}

\begin{example}{\rm The co-enveloping coalgebra $C^{\rm e}= C\ot C_{{\rm cop}}$ is a right $\times_C$-Hopf coalgebra where the source and target maps are given by
\begin{align*}
&\alpha: C\ot C_{{\rm cop}}\longrightarrow C, \quad c\ot c'\longmapsto c\ve (c'), \\
& \beta: C\ot C_{{\rm cop}}\longrightarrow C_{{\rm cop}}, \quad c\ot c'\longmapsto \ve(c) c',
\end{align*}

 multiplication, unite, and inverse maps  by
 \begin{align*}
&C\ot C_{{\rm cop}}\Co C\ot C_{{\rm cop}}\longrightarrow C\ot C_{{\rm cop}}, \quad c\ot c'\Co d\ot d'\longmapsto \ve(c) \ve(d')d\ot c' ,\\
&\eta: C\longrightarrow C\ot C_{{\rm cop}}, \quad c\longmapsto c\ps{2} \ot c\ps{1},\\
&\nu(c\ot c'\Co_C d\ot d')=\ve(c_4) c_1\ot c_2\ps{2}\Co_{C_{{\rm cop}}} c_2\ps{1}\ot c_3,\\
&\nu^{-1}(c\ot c'\Co_{C_{{\rm cop}}} d\ot d')=\ve(c_4) c_1\ps{1}\ot c_2\Co_C c_3 \ot c_1\ps{2}.
\end{align*}

}
\end{example}

The following example introduce a subcategory of $\times$-Hopf coalgebras. They are dual notion of  Hopf algebroids defined in \cite{b1}, hence we call them Hopf coalgebroids.
\begin{example}{\rm
  Let $C$ and $D$  be two coalgebras. A Hopf coalgebroid over the base coalgebras $C$ and $D$ is a triple $(\mathcal{K}, \mathcal{B}, S)$. Here $(\mathcal{K}, \Delta_{\mathcal{K}}, \alpha_{\mathcal{K}}, \beta_{\mathcal{K}}, \eta_{\mathcal{K}}, \mu_{\mathcal{K}})$ is a left $\times_C$-Hopf coalgebra and $(\mathcal{B}, \Delta_{\mathcal{B}}, \alpha_{\mathcal{B}}, \beta_{\mathcal{B}}, \eta_{\mathcal{B}}, \mu_{\mathcal{B}})$ is a right $\times_D$-Hopf coalgebra such that  as a coalgebra we have $\mathcal{K}=\mathcal{B}=H$ and there exists a $\mathbb{C}$-linear map $S: H\longrightarrow H$, called antipode. These structures are subject to the following axioms.
  \begin{itemize}
 \item[i)] $\beta_{\mathcal{B}} \circ \eta_{\mathcal{K}} \circ \alpha_{\mathcal{K}}=\beta_{\mathcal{B}},\quad \alpha_{\mathcal{B}} \circ \eta_{\mathcal{K}} \circ \beta_{\mathcal{K}}= \alpha_{\mathcal{B}}, \\
  \beta_{\mathcal{K}}\circ \eta_{\mathcal{B}} \circ \alpha_{\mathcal{B}}=\beta_{\mathcal{K}}, \quad \alpha_{\mathcal{K}}\circ \eta_{\mathcal{B}}\circ \beta_{\mathcal{B}}= \alpha_{\mathcal{K}}$.
 \item[ii)]$\mu_{\mathcal{B}}\circ (\mu_{\mathcal{K}}\Co_D \Id_H)= \mu_{\mathcal{K}}\circ (\Id_H\Co_C \mu_{\mathcal{B}}),\\
 \mu_{\mathcal{K}}\circ(\mu_{\mathcal{B}}\Co_C\Id_H)= \mu_{\mathcal{B}}\circ (\Id_H\Co_D \mu_{\mathcal{K}})$ .
 \item[iii)] $\beta_{\mathcal{K}}(S(h)\ps{1})\ot S(h)\ps{2}\ot \beta_{\mathcal{B}}(S(h)\ps{3})= \alpha_{\mathcal{K}}(h\ps{3})\ot S(h\ps{2})\ot \alpha_{\mathcal{B}}(h\ps{1}).$
 \item[iv)]$\mu_{\mathcal{K}}\circ (S\Co_C \Id_H)\circ \Delta_{\mathcal{K}}= \eta_{\mathcal{B}}\circ \alpha_{\mathcal{B}}, \quad \mu_{\mathcal{B}}\circ(\Id_H\Co_D S)\circ \Delta_{\mathcal{B}}= \eta_{\mathcal{K}}\circ \alpha_{\mathcal{K}}.$
 \end{itemize}
}\end{example}

\begin{example}\label{groupoid1}{\rm
Let $\Gc=(\Gc^1, \Gc^0)$ be a groupoid and $\Bc:=\Cb\Gc^1$ and  $C:=\Cb\Gc^0$ be the groupoid coalgebra as defined in Section \ref{section-1} with trivial coalgebra structure. Then via the source and target  maps $s,t:\Bc\ra C$ defined by $\a(g):=s(g)$, and $\b(g):=t(g)$,  where $s$ and $t$ are source and target maps  of $\Gc$, it is easy to verify that  $\Bc$ is a right $\times_C$-Hopf coalgebras where  the antipode is defined by
$$\nu^{-1}(g\ot h)= g\ot g^{-1}h.$$
}\end{example}

\begin{example}{\rm
It is shown in \cite{bm} that any weak  Hopf algebras  defines a $\times$-Hopf coalgebra as follows. First let us recall  that a weak bialgebra  $B$ is an unital algebra and counital coalgebra with the following compatibility conditions between the algebra and coalgebra  structures,
  \begin{align*}
   & (\Delta(1_B)\ot 1_B)(1_B\ot \Delta(1_B))=(\Delta\ot \Id_B)\circ \Delta(1_B)=\\
   &(1_B\ot \Delta(1_B))(\Delta(1_B)\ot 1_B),\\
    &\ve(b 1\ps{1})\ve(1\ps{2} b')= \ve(bb')= \ve(b 1\ps{2})\ve(1\ps{1} b').
  \end{align*}

  The first condition generalizes the  axiom of unitality of coproduct $\Delta$ and the third one generalizes the algebra map property of counit of   bialgebras. In fact its coproduct is not an unital map, \ie $\Delta(1_B)= 1\ps{1}\ot 1\ps{2}\neq 1_B\ot 1_B$. A weak Hopf algebra \cite{b1} is the weak bialgebra $B$ equipped with an antipode  $S: B\longrightarrow B$, satisfying the following conditions
  \begin{align*}
    &b\ps{1}S(b\ps{2})= \ve(1_B\ps{1}b)1_B\ps{2},\quad S(b\ps{1})b\ps{2}=1_B\ps{1}\ve(b 1_B\ps{2}), \\
     &S(b\ps{1})b\ps{2}S(b\ps{3})=S(b).
  \end{align*}
   Every weak Hopf algebra $B$ is a $\times_C$-Hopf coalgebra. Here $C= B/ \ker\xi$ where $\xi: B\longrightarrow B$ is given by $\xi(b)=\ve(1\ps{1}b)1\ps{2}.$ The source and target maps are given by
   \begin{align*}
     \alpha(b)=\pi(b), \qquad \beta(b)=\pi(S^{-1}(b)),
   \end{align*}
   where $\pi: B\longrightarrow C$ is the canonical projection map. The multiplication $\mu$ is given by the original multiplication  $B$ as an algebra. The unit map of the $\times_C$-Hopf coalgebra is given by $\eta= \xi$. Also we have $\nu(b\ot b')= bb'\ps{1}\ot b'\ps{2}$ and $\nu^{-1}(b\ot b')= bS(b'\ps{1})\ot b'\ps{2}.$
}\end{example}

\subsection{SAYD modules over $\times$-Hopf coalgebras}\label{SS-2-2}

In this subsection, we define modules, comodules,  and stable anti Yetter-Drienfeld (SAYD) modules over $\times$-Hopf coalgebras.
Then we present  some nontrivial examples of  stable anti Yetter-Drienfeld  modules over  the enveloping $\times_C$-Hopf coalgebra $C\ot C_{{\rm cop}}$. We also show that for a cyclic groupoid $(\Gc,\t)$ the coefficients $^\t C$ defined in Section \ref{section-1} defines  a SAYD module over the groupoid $\times$-Hopf coalgebra $\Cb\Gc^1$.    At the end we define  the  symmetries produced by $\times$-Hopf coalgebras on  algebras and coalgebras.

Let us bring a series of definitions that are   needed as preliminaries for  stable anti Yetter-Drienfeld (SAYD) modules over $\times$-Hopf coalgebras.
\begin{definition}
A right module $M$ over a left $\times_C$-Hopf coalgebra $\mathcal{K}$ is a right $C$-comodule where the action $\triangleleft:  M\Co_C \mathcal{K}  \longrightarrow M$ is a right $C$-comodule map. Here the right $C$-comodule structure of $M \Co_C\mathcal{K} $ is given by $m\ot k\longmapsto m\ot k\ps{2}\ot \beta({k\ps{1}}).$ Therefore the right $C$-comodule property of the action is equivalent to
\begin{equation}
  (m\triangleleft k)\sns{0}\ot (m\triangleleft k)\sns{1}=m\triangleleft k\ps{2}\ot \beta(k\ps{1}), \quad m\in M, k\in \mathcal{K}.
\end{equation}

\end{definition}

\begin{definition}[\cite{bal}]
  A left module $M$ over a left $\times_C$-Hopf coalgebras $\mathcal{K}$  is a left $C$-comodule where the action $  \mathcal{K} \Co_C M \longrightarrow M $ is a left $C$-comodule map where the left $C$-comodule structure of $\mathcal{K} \Co_C M$ is given by $k\ot m\longmapsto \alpha(k\ps{1}) \ot k\ps{2}\ot m.$
  \end{definition}

A left $\mathcal{K}$-module $M$  can be naturally  equipped with a right  $C$-comodule  structure by
\begin{equation}\label{coact}
  m\longmapsto  \eta(m\sns{-1})\ps{1}\triangleright m\sns{0} \ot  \alpha (\eta(m\sns{-1})\ps{2}).
\end{equation}
One checks that via these comodule structures,  $M$ becomes a $C$-bicomodule  and  the left $\mathcal{K}$-action is  a $C$-bicomodule map, \ie
\begin{equation}
  (k\triangleright m)\sns{-1}\ot (k\triangleright m)\sns{0}= \alpha(k\ps{1}) \ot k\ps{2}\triangleright m
\end{equation}
\begin{equation}
  (k\triangleright m)\sns{0}\ot (k\triangleright m)\sns{1}= k\ps{1}\triangleright m \ot \alpha(k\ps{2}).
\end{equation}
Furthermore with respect to the right $C$-coaction defined in \eqref{coact} one can use the $C$-comodule map property of $\mu$ and \eqref{2.8}  to show that  for  all $k\ot m  \in \Kc \Co_C M$ we have,
\begin{equation}
  k\triangleright m\sns{0} \ot m\sns{1}= k\ps{1}\triangleright m\ot \beta(k\ps{2}).
\end{equation}
For later purposes, one should note that any $\Kc$-module map is proved to be $C$-bicolinear.

\begin{definition}
  Let  $\mathcal{K}$ be a left $\times_C$-Hopf coalgebra. A right $\mathcal{K}$-comodule $M$ is a  right $\Kc$-comodule, $\rho: M\ra M\ot \Kc$,   where $\Kc$ is considered as a coalgebra.
    A right $\mathcal{K}$-comodule $M$ can be naturally equipped with a $C_{\rm cop}$-bicomodule structure where the right $C_{\rm cop}$-coaction is given by
  \begin{equation}\label{induced-coaction}
    m\sns{0}\ot m\sns{1}:= m\ns{0}\ot \beta(m\ns{1}),
  \end{equation}
  and the left $C_{\rm cop}$-coaction by
  \begin{align}
    m\sns{-1}\ot m\sns{0}:= \alpha(m\ns{1})\ot m\ns{0}.
  \end{align}
     One observes that the map  $\rho$ actually lands in $M\Co_{C_{\rm cop}} \Kc$, and the induced  map $\rho: M\longrightarrow M \Co_{C_{\rm cop}} \mathcal{K}$, is a coassociative,  counital, and  $C$-bicomodule map.
\end{definition}
The $C_{\rm cop}$-bicomodule structure of $M$ is a result of the cocommutativity of the ranges of the maps $\alpha$ and $\beta$.

\begin{definition}
 Let  $\mathcal{K}$  be a left $\times_C$-Hopf coalgebra. A left $\mathcal{K}$-comodule $M$ is a  left $\mathcal{K}$-comodule, $\rho:M\longrightarrow \mathcal{K}\ot M$, where $\mathcal{K}$ is considered as a coalgebra.
  A left $\mathcal{K}$-comodule $M$ can be naturally equipped with a $C_{\rm cop}$-bicomodule structure by
  \begin{align}
    &\rho^R_{C}(m):= m\ns{0}\ot \alpha(m\ns{-1}),\quad \rho_C^L(m):= \beta(m\ns{-1})\ot m\ns{0}.
  \end{align}

    The map  $\rho$  lands in $\Kc \Co_{C_{\rm cop}} M$, and the induced  map $\rho: M\longrightarrow \Kc \Co_{C_{\rm cop}} M$  is a coassociative,  counital, and  $C$-bicomodule map.
\end{definition}

Now we are ready to define stable anti Yetter-Drinfeld modules over  $\times$-Hopf coalgebras.
\begin{definition}
Let  $\mathcal{K}$  be a left $\times_C$-Hopf coalgebra, $M$ a left $\mathcal{K}$-module and a right $\mathcal{K}$-comodule. We call $M$ a
left-right anti Yetter-Drinfeld module over $\mathcal{K}$ if

 i) The right coaction of $C$ induced on $M$ via \eqref{induced-coaction} coincides with the canonical coaction defined in \eqref{coact}.\\

ii) For all $ m \ot k \in M \Co_{C_{\rm cop}} \Kc$, we have
\begin{equation}\label{SAYD}
(  k\triangleright m)\ns{0} \ot (  k\triangleright m)\ns{1}= \nu^+(m\ns{1}, k\ps{1})\triangleright m\ns{0} \ot k\ps{2} \nu^-(m\ns{1}, k\ps{1}).
\end{equation}
We call $M$  stable if $m\ns{1}m\ns{0}=m$.
\end{definition}
One notes that the AYD condition \eqref{SAYD} is well-defined as $ m \ot k \in M \Co_{C_{\rm cop}} \Kc$ yields
\begin{align*}
  m\ns{0}\ot \beta(m\ns{1})\ot k= m\ot \beta(k\ps{1})\ot k\ps{2}.
\end{align*}
By applying $ \blacktriangledown \ot \Id \ot \Delta$ on the both hand sides of the previous equation and using the coassociativity of the right coaction of $\Kc$ on $M$ we have
\begin{align*}
  &m\ns{0}\ot m\ns{1}\ps{1}\ot \beta(m\ns{1}\ps{2})\ot k\ps{1}\ot k\ps{2}= \\
  &m\ns{0}\ot m\ns{1}\ot \beta(k\ps{1})\ot k\ps{2}\ot k\ps{3}.
\end{align*}
This shows that
\begin{align*}
  m\ns{0}\ot m\ns{1}\Co_{C_{\rm cop}} k\ps{1}\ot k\ps{2},
\end{align*}
 and therefore \eqref{SAYD} is well-defined. This definition generalizes the definition of the stable anti Yetter-Drinfeld modules over Hopf algebras \cite[Definition 4.1]{HaKhRaSo1}.

 One similarly defines modules and comodules over right $\times$-Hopf coalgebras.

\begin{definition}

 Let $\mathcal{B}$  be a right $\times_C$-Hopf coalgebra, $M$ be a right $\mathcal{B}$-module and a left $\mathcal{B}$-comodule. We call $M$ a
right-left anti Yetter-Drinfeld module over $\mathcal{B}$ if

i)
The left coaction of $C$ on $M$ is the same as the following canonical coaction \begin{equation}\label{c-action}
  m\longmapsto \alpha(\eta(m\sns{1})\ps{1})\ot m\sns{0}\triangleleft \eta(m\sns{1})\ps{2}.
\end{equation}

ii) For all $ b\ot m \in \Bc \Co_{C_{\rm cop}} M$ we have
\begin{equation}\label{SAYD2}
(m \triangleleft b)\nsb{-1} \ot (  m \triangleleft b)\nsb{0}= \nu^+(b\ps{2}, m\nsb{-1}) b\ps{1}\ot m\nsb{0}\triangleleft \nu^-(b\ps{2}, m\nsb{-1}).
\end{equation}
We call $M$  stable if $m\nsb{0} m\nsb{-1}=m$.
\end{definition}
One notes that the AYD condition \eqref{SAYD2} is well defined because the following term in cotensor is well-defined.
\begin{equation*}
  b\ps{1}\ot b\ps{2} \Co_{C_{\rm cop}} m\nsb{-1} \ot m\nsb{0}.
\end{equation*}
The following object will help us to introduce an example of SAYD module for $\times$-Hopf coalgebras.
\begin{definition}
A right group-like  of a left $\times_C$-Hopf coalgebra $\mathcal{K}$ is a linear map $\delta: C\longrightarrow \mathcal{K}$ satisfying the following conditions\\
\begin{align}\label{char-1}
 & \delta(c)\ps{1}\ot \alpha(\delta(c)\ps{2})= \delta(c\ps{1})\ot c\ps{2},\\ \label{char-2}
 & \delta(\alpha(\delta(c)\ps{1}))\ot \delta(c)\ps{2}= \delta(c)\ps{1}\ot \delta(c)\ps{2},\\ \label{char-3}
 & \ve(\delta(c))=\ve(c).
 \end{align}
\end{definition}
\begin{example}{\rm
   The unit map $\eta$ is a right group-like of a right $\times_C$-Hopf coalgebra. The above three conditions for $\eta$ are equivalent to \eqref{bos}, \eqref{etaalphaeta2} and \eqref{epsil}, respectively. One also notes that for a Hopf algebra the above  definition reduces to the original definition of a group-like element.}
 \end{example}
 \begin{definition}
   A character of  a $\times_C$-Hopf coalgebra $\mathcal{K}$, left or right, is a ring morphism $\sigma:\mathcal{K}\longrightarrow C,$ where  $\sigma(\eta_{\mathcal{K}}(c))=c$. Here $C$ is considered to be a ring on itself, \ie  the product $C\Co_C C\longrightarrow C$ is given by $x\ot y\longmapsto \ve(x)y.$
 \end{definition}
 As an example, the source and target maps $\alpha$ and $\beta$ are characters of the $\times_C$-Hopf coalgebra $\mathcal{K}= C\ot C_{{\rm cop}}$.

  Let $\mathcal{K}$ be  a left $\times_C$-Hopf coalgebra, $\delta: C\longrightarrow \mathcal{K}$ a right group-like,  and $\sigma:\mathcal{K}\longrightarrow C$ a character  of $\mathcal{K}$. We define the action and  a coaction of $\Kc$ on $C$ as follows,
 \begin{equation}\label{actcoact}
   c\longmapsto \alpha(\delta(c)\ps{1})\ot \delta(c)\ps{2}, \quad k\triangleright c= \alpha(k\ps{1})\ve(c)\ve(\sigma(k\ps{2})).
 \end{equation}

\begin{lemma}
The action defined in \eqref{actcoact} is associative.
\end{lemma}
\begin{proof}
 For any $k_1,k_2\in\mathcal{K}$ and $c\in C$ we have,
\begin{align*}
 &k_1 k_2\triangleright c= \alpha(k_1 k_2)\ps{1})\ve(c) \ve(\sigma(k_1 k_2)\ps{2}))\\
 & =\alpha(k_1\ps{1} k_2\ps{1})\ve(c) \ve(\sigma(k_1\ps{2} k_2\ps{2}))\\
 &=\alpha(k_1\ps{1})\ve(k_2\ps{1})\ve(c)\ve(\sigma(k_1\ps{2}))\ve(\sigma(k_2\ps{2})) \\
 &=\alpha(k_1\ps{1})\ve(\alpha(k_2\ps{1}))\ve(c)\ve(\sigma(k_1\ps{2}))\ve(\sigma(k_2\ps{2}))\\
 &= k_1\triangleright [\alpha(k_2\ps{1})\ve(c)\ve(\sigma(k_2\ps{2})]=k_1\triangleright (k_2\triangleright c).
\end{align*}
We use  \eqref{brzmil2} in the second equality. We apply  Lemma \ref{leftbi}(iii), the multiplicity of the ring map $\sigma$,  and the multiplication rule in $C$   in the third equality. Finally we use  the counitality of $\alpha$ in the fourth equality.
\end{proof}
\begin{lemma}
The coaction defined in \eqref{actcoact} is coassociative and counital.
\end{lemma}
\begin{proof}
It is straightforwardly seen that,
 \begin{align*}
 &c\ns{0}\ot c\ns{1}\ps{1}\ot c\ns{1}\ps{2} =\alpha(\delta(c)\ps{1})\ot \delta(c)\ps{2}\ot \delta(c)\ps{3}\\
 &=\alpha(\delta(\alpha(\delta(c)\ps{1}))\ps{1})\ot \delta(\alpha(\delta(c)\ps{1}))\ps{2}\ot \delta(c)\ps{2}= c\ns{0}\ns{0}\ot c\ns{0}\ns{1}\ot c\ns{1}.
\end{align*}
Here we  use \eqref{char-2} in the second equality. The counitality of the coaction comes as follows.
\begin{align*}
  \ve(\delta(c)\ps{2})\alpha(\delta(c)\ps{1})=\alpha(\delta(c))=c.
\end{align*}
For the last equality, we apply $\ve\ot \Id$ on both hand sides of \eqref{char-1} and then we use  \eqref{char-3}.
\end{proof}
\begin{proposition} \label{Proposition-SAYD-on-C}The action and coaction defined in \eqref{actcoact} amount to an AYD module over  $\mathcal{K}$ if and only if
\begin{align}\notag
  &\ve(\sigma(\nu^+\ps{2}(\delta(c)\ot k)))\ve(\nu^-(\delta(c)\ot k))\alpha(\nu^+\ps{1}(\delta(c)\ot k)) \\\label{ayd11}
  &=\ve(c)\ve(\sigma(k\ps{2}))\alpha(\delta(\alpha(k\ps{1}))),
\end{align}
and
\begin{align}
  \alpha(\eta(c)\ps{2})\ve(\sigma(\eta(c)\ps{1}))=\alpha(\delta(c)).
\end{align}
It is stable if and only if
\begin{equation}
  \alpha(\delta(c)\ps{1})\ve(\sigma(\delta(c)\ps{2}))=c.
\end{equation}

\end{proposition}
\begin{proof}
One notes that the  right $\mathcal{K}$-comodule structure of $C$ is given by $c\longmapsto c\ps{1}\ot \delta(c\ps{2})$.
\end{proof}
 We specialize  Proposition \ref{Proposition-SAYD-on-C} to the left $\times_C$-Hopf coalgebra $\mathcal{K}= C\ot C_{{\rm cop}}$ whose structure is given in  the Example \ref{example}.
 We first investigate  right group-like morphisms of $C\ot C_{{\rm cop}}$. We claim that a map $\delta: C\longrightarrow C\ot C_{{\rm cop}}$ is a right group-like if and only if $\delta(c)=c\ps{2}\ot \theta(c\ps{1})$ for some counital anti-coalgebra map $\theta: C \longrightarrow C$.
To prove the claim, let $\delta$ be a right group-like and $\delta(c)= \sum_i a_i\ot b_i$. We have
 \begin{align*}
 &\delta(c)= \sum_i a_i\ot b_i = \sum_i \ve(b_i\ps{1})a_i\ps{2}\ot \ve(a_i\ps{1})b_i\ps{2}=\\
 & \sum_i \alpha(a_i\ps{2}\ot b_i\ps{1})\ot \beta(a_i\ps{1}\ot b_i\ps{2})=\alpha(\delta(c)\ps{2})\ot \beta(\delta(c)\ps{1})=c\ps{2}\ot \beta(\delta(c\ps{1})),
 \end{align*}
where we use \eqref{char-1} in the last equality. Let us set $\theta= \beta \circ \delta$. The map $\theta$ is obviously an anti-coalgebra map. It is counital because,
 \begin{equation*}
 \ve(\theta(c))=\ve(\beta(\delta(c)))=\ve(\beta(\sum_i a_i\ot b_i))=\sum_i \ve(\ve(a_i)b_i)=\sum_i \ve(a_i)\ve(b_i)=\ve(c).
 \end{equation*}
 Conversely, let $\theta: C\longrightarrow C$ be an anti-coalgebra map. One defines a group-like  $\delta:C\longrightarrow C\ot C_{{\rm cop}}$ by
 \begin{equation} c\longmapsto c\ps{2} \ot \theta(c\ps{1}).
 \end{equation}
  Indeed,  the following computations show that  $\delta$ satisfies \eqref{char-1}, \eqref{char-2}, and  \eqref{char-3}. First we  use the definition of $\alpha $ and counitality of $\theta$ to see
 \begin{align*}
   &\delta(c)\ps{1}\ot \alpha(\delta(c)\ps{2})= c\ps{3}\ot \theta(c\ps{2})\ot \alpha(c\ps{4}\ot \theta(c\ps{1})) \\
   &=c\ps{3}\ot \theta(c\ps{2})\ot c\ps{4}\ve(c\ps{1}))=c\ps{2}\ot \theta(c\ps{1})\ot c\ps{3}= \delta(c\ps{1})\ot c\ps{2}.
 \end{align*}
  To see \eqref{char-2}, we have
 \begin{align*}
   &\delta(\alpha(\delta(c)\ps{1}))\ot \delta(c)\ps{2}= \delta(\alpha(c\ps{3}\ot \theta(c\ps{2})))\ot c\ps{4}\ot \theta(c\ps{1})\\ &=\delta(c\ps{3}\ve(\theta(c\ps{2})))\ot c\ps{4}\ot \theta(c\ps{1})=\delta(c\ps{3}\ve(c\ps{2}))\ot c\ps{4}\ot \theta(c\ps{1})\\
   &=\delta(c\ps{3}\ve(c\ps{2}))\ot c\ps{4}\ot \theta(c\ps{1})= c\ps{3}\ot \theta(c\ps{2})\ot c\ps{4} \ot \theta(c\ps{1})= \delta(c)\ps{1}\ot \delta(c)\ps{2},
 \end{align*}
where the counitality of $\theta$ is used in the third equality. Also
\begin{equation*}
  \ve(\delta(c))= \ve(c\ps{2}\ot \theta(c\ps{1}))= \ve(c\ps{2}) \ve(\theta(c\ps{1}))= \ve(c\ps{2})\ve(c\ps{1})= \ve(c),
\end{equation*}
shows that $\delta$ satisfies \eqref{char-3}, where the counitality of $\theta$ is used in the third equality. Therefore $\delta$ is a right group-like.
Now we find all characters  $\sigma: C\ot C_{{\rm cop}}\longrightarrow C$. Since $\sigma$ is a ring morphism, for $c,d\in C$ and $c',d'\in C_{{\rm cop}}$ we have
\begin{align}\label{abc1}
 \ve(\sigma(c\ot c'))\sigma(d\ot d')= \sigma(\mu(c\ot c'\Co_C d\ot d'))=\ve(c')\ve(d)\sigma(c\ot d').
\end{align}
For $\mathcal{K}=C\ot C_{{\rm cop}}$, the condition $\sigma(\eta_{\mathcal{K}}(c))=c$  is equivalent to
\begin{align}\label{abc22}
\sigma(c\ps{1}\ot c\ps{2})=c, \quad c\in C.
\end{align}
Let us set $c\ot c'= a\ps{1}\ot a\ps{2}$ in \eqref{abc1}. Using \eqref{abc22}, we obtain
\begin{align}
  \ve(a)\sigma(d\ot d')=\ve(d)\sigma(a\ot d').
\end{align}
If in the preceding equation we put $d\ot d'= b\ps{1}\ot b\ps{2}$, we have
\begin{align}\label{abc44}
   \sigma(a\ot b)=\ve(a)b, \quad a,b\in C.
\end{align}
Conversely, one easily sees that any $\mathbb{C}$-linear morphism $\sigma: C\ot C_{{\rm cop}}\longrightarrow C$ satisfying \eqref{abc44} is a character.

\begin{proposition}\label{prop-sayd-top}
  Let $\mathcal{K}$ be the left $\times_C$-Hopf coalgebra $C\ot C_{{\rm cop}}$, $\delta$ a right group-like,  and $\sigma$ a character of $\mathcal{K}$. Then the following action and coaction
  \begin{align*}
  c\longmapsto c\ps{2}\ot c\ps{3}\ot \theta(c\ps{1}), \quad (c_1\ot c_2)\triangleright c_3= c_1 \ve(c_2) \ve(c_3),
\end{align*}
define a left-right $\Kc$-SAYD module on $C$.
\end{proposition}
\begin{proof}
 By the above  characterization of right group-like morphisms of  $\mathcal{K}$ the action and coaction defined in \eqref{actcoact} reduce to
\begin{align*}
  c\longmapsto c\ps{2}\ot c\ps{3}\ot \theta(c\ps{1}), \quad (c_1\ot c_2)\triangleright c_3= c_1 \ve(c_2) \ve(c_3).
\end{align*}
 Let us see the above claim in details,
\begin{align*}
  &c\longmapsto  \alpha(c\ps{3}\ot \theta(c\ps{1}))\ot c\ps{4}\ot \theta(c\ps{2})\\
  &=c\ps{3}\ve(c\ps{1})\ot c\ps{4}\ot \theta(c\ps{2})= c\ps{2}\ot c\ps{3}\ot \theta(c\ps{1}).
\end{align*}
Using \eqref{abc44},   we see
\begin{align*}
 & (c_1\ot c_2)\triangleright c_3= \alpha(c_1\ps{1}\ot c_2\ps{2})\ve(c_3)\ve(\sigma(c_1\ps{2}\ot c_2\ps{1}))\\
 &=c_1\ps{1}\ve(c_2\ps{2})\ve(c_3)\ve(\sigma(c_1\ps{2}\ot c_2\ps{1}))=c_1\ps{1}\ve(c_3)\ve(\ve(c_1\ps{2})c_2)= c_1\ve(c_3)\ve(c_2).
\end{align*}

\end{proof}
\begin{remark}\label{groupoid2}
Let $\Gc=(\Gc^1,\Gc^0)$ be a groupoid endowed with a cyclic structure $\t$. We have seen in the Example \ref{groupoid1} that $\Bc:=\Cb\Gc^1$ is a right $\times_C$-Hopf coalgebra, where $C:=\Cb\Gc^0$. We observe  that the right module left comodule $^\t C$,  defined in  \eqref{act-group} and \eqref{coact-group} respectively,  defines a  SAYD module structure on $C$. Indeed,  the stability condition is obvious since for all $c\in C$ we have
  \begin{equation*}
    c\cdot \theta_c= \ve(c)s(\theta_c)=c.
  \end{equation*}
  Since   $c\ot g\in C\Co_C \Bc$,  using \eqref{last1} and \eqref{last2} we obtain $s(g)=c$. The following computation shows that the AYD condition holds.
  \begin{align*}
    &(c \cdot g)\ns{-1}\ot (c \cdot g)\ns{0}=\ve(c)s(g)\ns{-1}\ot s(g)\ns{0}=\ve(c)\theta_{s(g)}\ot c\\
    &= g^{-1}\theta_c g\ot \ve(c)s(g)= \nu^+(g, \theta_c)g\ot c\cdot \nu^-(g, \theta_c).
  \end{align*}
  We use   $s(g)=c$ and $\theta_c g= g \theta_{s(g)}$ in the third equality and \eqref{last3} in the last equality.
\end{remark}


\subsection{Symmetries via bicoalgebroids}
In this subsection we provide  some  preliminaries  which  are used in the sequel sections. We define  the following three symmetries for $\times$-Hopf coalgebras.
\begin{definition}
 Let  $\mathcal{K}$ be  a left $\times_{C}$-Hopf coalgebra. A left $\mathcal{K}$-comodule coalgebra $T$  is a coalgebra and a left $\mathcal{K}$-comodule such that for all  $t\in T$ the following identities hold,
\begin{align}\label{comodulecoalgebra-1}
& t\ns{-1}\ve_T( t\ns{0})= \eta(\alpha(t\ns{-1}))\ve_T( t\ns{0}),\\ \label{comodulecoalgebra-2}
 &t\ns{-1}\ot t\ns{0}\ps{1}\ot t\ns{0}\ps{2}=t\ps{1}\ns{-1} t\ps{2}\ns{-1} \ot t\ps{1}\ns{0}\ot t\ps{2}\ns{0} ,\\ \label{comodulecoalgebra-3}
 & \beta(t\ps{1}\ns{-1})\ot t\ps{1}\ns{0}\ot t\ps{2}= \alpha(t\ps{2}\ns{-1})\ot t\ps{1}\ot t\ps{2}\ns{0}.
\end{align}
\end{definition}
The equation \eqref{comodulecoalgebra-3} is stating that  the comultiplication of $T$ is $C$-cobalanced.


 Recall that a $C$-ring $A$  is a  $C$-bicomodule endowed with two $C$-bicolinear operators, the multiplication $\mathfrak{m}: A\Co_C A\longrightarrow A$ and unit map $\eta_A: C\longrightarrow A$ where the multiplication  is associative  and the unit map is unital, \ie
 \begin{equation}
   \mathfrak{m}(a\sns{0}\ot \eta_A(a\sns{1}))=a= \mathfrak{m}(\eta_A(a\sns{-1})\ot a\sns{0}).
 \end{equation}
\begin{definition}
 Let $\mathcal{B}$ be a right $\times_{C}$-Hopf coalgebra. A right $\mathcal{B}$-comodule ring $A$  is a $C$-ring and a right $\mathcal{B}$-comodule satisfying the following conditions:
\begin{align}\label{mc1}
&\eta_A(c)\nsb{0}\ot \eta_A(c)\nsb{1}= \eta_A(\alpha(\eta_{\mathcal{B}}(c)\ps{1}))\ot \eta_{\mathcal{B}}(c)\ps{2},\\\label{mc2}
&\mathfrak{m}(a_1 \ot a_2)\nsb{0}\ot \mathfrak{m}(a_1 \ot a_2)\nsb{1}=\mathfrak{m}(a_1\nsb{0} \ot a_2\nsb{0})\ot a_1\nsb{1} a_2\nsb{1}.
\end{align}
\end{definition}
As an example, $\mathcal{B}$ is a right $\mathcal{B}$-comodule ring where $\mathfrak{m}=\mu_{\mathcal{B}}$ and the right comodule structure is given by $\Delta_{\mathcal{B}}$. In this case the relations \eqref{mc1} and \eqref{mc2} are equivalent to \eqref{brzmil23} and \eqref{etaalphaeta2} in the definition of right bicoalgebroids.

\smallskip

The following definition is needed  do define   $\times$-Hopf coalgebra Galois coextensions in the last section.

\begin{definition}\label{module-coalgebra}
  Let $\mathcal{B}$ be a right $\times_C$-Hopf coalgebra and the coalgebra $T$ be a right $\Bc$-module, a $C$-bicomodule,  and a $S-C$-bicomodule such that the comultiplication of $T$ is a morphism of right $C$-comodules \ie
  $$t\ps{1}\ot t\ps{2}\sns{0}\ot t\ps{2}\sns{1}=t\sns{0}\ps{1}\ot t\sns{0}\ps{2}\ot t\sns{1}.$$
   We say $T$ is a  $\mathcal{B}$-module coalgebra  if

\begin{enumerate}
\item[i)] $t\ps{1}\sns{0}\ot t\ps{1}\sns{1}\ot t\ps{2}= t\ps{1}\ot t\ps{2}\sns{-1}\ot t\ps{2}\sns{0}, $
\item[ii)] $\ve_T(t\triangleleft b)= \ve_T (t) \ve_{\mathcal{B}}(b),$
\item[iii)]$(t \triangleleft b)\ps{1}\ot (t \triangleleft b)\ps{2}= t\ps{1}\triangleleft b\ps{1} \ot t\ps{2}\triangleleft b\ps{2} .$
\end{enumerate}
\end{definition}
The condition $(i)$ means that the comultiplication of $T$ is $C$-cobalanced. One notes that the right hand side of the condition (iii) is well-defined by the fact that the comultiplication of $T$ is a morphism of right $C$-comodules.

\smallskip
 As an example any right $\mathcal{B} $-Hopf coalgebra is a right $\mathcal{B}$-module coalgebra where the action of $\mathcal{B}$ on itself is defined by the multiplication $\mu_{\mathcal{B}}$. The conditions (i), (ii) and (iii) are equivalent to the $C$-cobalanced property of $\mathcal{B}$ which comes from $C$-bicomodule structure of $\mathcal{B}$, \eqref{var2} and \eqref{brzmil23} equivalently.


\section{Hopf cyclic cohomology of $\times$-Hopf coalgebras}
In this section we introduce the cyclic cohomology  of  comodule coalgebras and comodule rings with coefficients in SAYD modules under the symmetry of  $\times$-Hopf coalgebras.
\subsection{Cyclic cohomology of comodule coalgebras }
\label{SS-2-3}

 Let $\mathcal{K}$ be a left $\times_S$-Hopf coalgebra, $T$ a left $\mathcal{K}$-comodule coalgebra,  and $M$ a left-right SAYD module over $\mathcal{K}$. We set
$$^\mathcal{K}C^n(T, M)=M\Co_{\mathcal{K}}T^{\Co_{S}(n+1)}.$$ We define the following cofaces, codegeneracies,  and cocyclic map.
\begin{align}
\begin{split}\label{cocyclic-module}
&d_i(m\ot \widetilde{t})= m\ot t_0 \ot\cdots \ot \Delta(t_i)\ot \cdots \ot t_n, \qquad 0\leq i \leq n-1,\\
&d_n(m\ot \widetilde{t})=\\
 &t_0\ps{2}\ns{-1} t_1\ns{-1}\cdots t_n\ns{-1}\triangleright m\ot t_0\ps{2}\ns{0}\ot t_1\ns{0}\ot \cdots \ot t_n\ns{0} \ot t_0\ps{1},  \\
&s_i(m\ot \widetilde{t})= m\ot t_0 \ot\cdots \ot \ve(t_i)\ot \cdots \ot t_n,\qquad \quad \qquad 0\leq i \leq n,\\
&\mathfrak{t}_n(m\ot \widetilde{t})= t_1\ns{-1} \cdots  t_n\ns{-1}\triangleright m\ot t_1\ns{0}\ot \cdots \ot t_n\ns{0}\ot t_0,
\end{split}
\end{align}
where $\widetilde{t}=t_0\ot\cdots\ot t_n$. The left $\mathcal{K}$-comodule structure of $T^{\Co_S(n+1)}$ is given by
\begin{align}
& t_0\ot \cdots \ot t_n\longmapsto  t_0\ns{-1}\cdots  t_{n}\ns{-1}\ot t_0\ns{0}\ot t_1\ns{0}\ot \cdots \ot t_n\ns{0}
\end{align}
\begin{proposition}
 The morphisms defined in  \eqref{cocyclic-module}  turn  $^\mathcal{K}C^n(T, M)$ to a cocyclic module.
\end{proposition}

\begin{proof}
We leave  to the reader to check that the operators satisfy the commutativity relations for a cocyclic module \cite{NCG}. However,  we check that the operators are well-defined; this is obvious  for the degeneracies and all faces except possibly the very last one. Based on the relations in the cocyclic category,  it suffices to check that  the cyclic operator is well-defined. Indeed, for $n=0$, the map $\tau$ is identity map and therefore is well-defined.

 The following computation shows that the cyclic map is well-defined in general.
\begin{align*}
&(t_1\ns{-1} \cdots  t_n\ns{-1}\triangleright m)\ns{0}\ot (t_1\ns{-1} \cdots  t_n\ns{-1}\triangleright m)\ns{1}\ot t_1\ns{0}\ot \cdots\\
&\cdots \ot t_n\ns{0}\ot t_0\\
  &=\nu^+[m\ns{1}, t_1\ns{-1}\ps{1} \cdots  t_n\ns{-1}\ps{1}]\triangleright m\ns{0}\ot t_1\ns{-1}\ps{2} \cdots t_{n-1}\ns{-1}\ps{2} t_n\ns{-1}\ps{2} \\
  &\nu^-(m\ns{1}, t_1\ns{-1}\ps{1} \cdots t_{n-1}\ns{-1}\ps{1} t_n\ns{-1}\ps{1})\ot t_1\ns{0}\ot \cdots \ot t_n\ns{0}\ot t_0\\
   &=\nu^+[t_0\ns{-1}\cdots  t_{n}\ns{-1},    t_1\ns{0}\ns{-1}\ps{1} \cdots  t_n\ns{0}\ns{-1}\ps{1}]\triangleright m\ot \\
  &t_1\ns{0}\ns{-1}\ps{2} \cdots t_{n-1}\ns{0}\ns{-1}\ps{2} t_n\ns{0}\ns{-1}\ps{2}\nu^-[t_0\ns{-1}\cdots t_{n}\ns{-1}, t_1\ns{0}\ns{-1}\ps{1} \cdots\\
  &\cdots   t_n\ns{0}\ns{-1}\ps{1}] \ot t_1\ns{0}\ns{0}\ot \cdots \ot t_n\ns{0}\ns{0}\ot t_0\ns{0}\\
   &=\nu^+[t_0\ns{-1}t_1\ns{-1}\ps{1}\cdots t_{n-1}\ns{-1}\ps{1} t_{n}\ns{-1}\ps{1}, t_1\ns{-1}\ps{2}\ps{1} \cdots \\
   &\cdots t_{n-1}\ns{-1}\ps{2}\ps{1} t_n\ns{-1}\ps{2}\ps{1}]\triangleright m\ot \\
  &t_1\ns{-1}\ps{2}\ps{2} \cdots  t_n\ns{-1}\ps{2}\ps{2}\nu^-[t_0\ns{-1}t_1\ns{-1}\ps{1}\cdots  t_{n}\ns{-1}\ps{1})), t_1\ns{-1}\ps{2}\ps{1} \cdots\\
    &\cdots t_n\ns{-1}\ps{2}\ps{1}]\ot t_1\ns{0}\ot \cdots \ot t_n\ns{0}\ot t_0\ns{0}\\
  &=\nu^+[t_0\ns{-1}t_1\ns{-1}\ps{1}\cdots t_{n}\ns{-1}\ps{1}, t_1\ns{-1}\ps{2} \cdots  t_n\ns{-1}\ps{2}]\triangleright m\ot \\
  &t_1\ns{-1}\ps{3} \cdots  t_n\ns{-1}\ps{3}\nu^-[t_0\ns{-1}t_1\ns{-1}\ps{1}\cdots t_{n}\ns{-1}\ps{1}, t_1\ns{-1}\ps{2}\cdots  t_n\ns{-1}\ps{2}] \\
  &\ot t_1\ns{0}\ot \cdots \ot t_n\ns{0}\ot t_0\ns{0}\\
  &=\nu^+[t_0\ns{-1}\ot \{t_1\ns{-1}\cdots t_{n}\ns{-1}\}\ps{1}, \{t_1\ns{-1} \cdots  t_n\ns{-1}\}\ps{2}]\triangleright m\ot \\
  &\{t_1\ns{-1} \cdots  t_n\ns{-1}\}\ps{3}\nu^-[t_0\ns{-1}\{t_1\ns{-1}\cdots  t_{n}\ns{-1}\}\ps{1}, \{t_1\ns{-1} \cdots t_n\ns{-1}\}\ps{2}] \\
  &\ot t_1\ns{0}\ot \cdots \ot t_n\ns{0}\ot t_0\ns{0}\\
  & =t_1\ns{-1}\ps{1} \cdots  t_n\ns{-1}\ps{1}\triangleright m \ot t_1\ns{-1}\ps{2} \cdots t_{n}\ns{-1}\ps{2} t_0\ns{-1}\ot t_1\ns{0}\ot \cdots \\
  &\cdots\ot t_n\ns{0}\ot t_0\ns{0}\\
  & =t_1\ns{-1} \cdots  t_n\ns{-1}\triangleright m \ot t_1\ns{0}\ns{-1} \cdots t_{n}\ns{0}\ns{-1} t_0\ns{-1}\\
 &\ot t_1\ns{0}\ns{0}\ot \cdots \ot t_n\ns{0}\ns{0}\ot t_0\ns{0}.
\end{align*}
We use AYD condition and comultiplicative property \eqref{brzmil2} in the first equality. For the second equality, we use $m\ot\widetilde{t}\in M\Co_{\mathcal{K}}T^{\Co_S(n+1)}$ that is  equivalent to \begin{align*}
&m\ns{0}\ot m\ns{1}\ot t_0\ot\cdots \ot t_n\\
&~~~~~~~~~~~~~~~~~~= m\ot t_0\ns{-1}\cdots  t_{n}\ns{-1}\ot t_0\ns{0}\ot t_1\ns{0}\cdots \ot t_n\ns{0}.
\end{align*}
We use the comodule property for the elements $t_1, \cdots, t_n$ in the third equality, comultiplicative property \eqref{brzmil2} in the fifth equality, Lemma \eqref{property}(ii) in the sixth equality and the comodule property for the elements $t_1, \cdots, t_n$ in the last equality.
\end{proof}
The cyclic cohomology of the preceding  cocyclic module will be denoted by $^{\mathcal{K}}HC^*(T, M)$.
It should be noted that even for the case  $\Kc:=H$ is merely a Hopf algebra and $T=H$ coacting on itself via adjoint coaction, the above cocyclic  module did not appear in the literature. Let us simplify  the above cocyclic module in  the case of right $\times$-Hopf coalgebras for Hopf algebras.
We recall that for a Hopf algebra $H$ a right comodule coalgebra $C$ is a right $H$-comodule and a coalgebra such that
\begin{equation*}
   c\ns{0}\ps{1}\ot c\ns{0}\ps{2}\ot c\ns{1}=  c\ps{1}\ns{0} \ot c\ps{2}\ns{0}\ot c\ps{1}\ns{1} c\ps{2}\ns{1}, \quad \ve(c\ns{0})a\ns{1}=\ve(a)1_H.
\end{equation*}
As an example, $H$ is a right $H$-comodule coalgebra by the coadjoint coaction $h\longmapsto h\ps{2}\ot S(h\ps{1})h\ps{3}$.

\begin{proposition}
  Let  $H$ be a Hopf algebra, $C$ a right $H$-comodule coalgebra,  and $M$ a right-left SAYD module on $H$. We set
   \begin{equation}\label{module02}
     ^{H}\mathcal{C}^{n}(C, M)= C^{\ot(n+1)}\Co_H M.
   \end{equation}
   The following cofaces, codegeneracies,  and cocyclic maps define a cocyclic module for $^{H}C^{n}(\mathcal{C}, M)$.
   \begin{align}\label{cocyclic02}\nonumber
    &\d_i( c_0\ot \cdots \ot c_n\ot m)=  c_0\ot\cdots \ot \Delta(c_i)\ot \cdots\ot c_n\ot m,\\ \nonumber
    &\d_n( c_0\ot \cdots \ot c_n\ot m)=c_0\ps{2}\ot c_1\ot\cdots\ot c_n\ot c_0\ps{1}\ns{0}\ot m\triangleleft c_0\ps{1}\ns{1},\\ \nonumber
    &\s_i( c_0\ot \cdots \ot c_n\ot m)= c_0\ot\cdots \ot \varepsilon(c_{i+1})\ot \cdots\ot c_n\ot m,\\ \nonumber
    &\tau_n( c_0\ot \cdots \ot c_n\ot m)= c_1\ot \cdots \ot c_n\ot c_0\ns{0}\ot m\triangleleft c_0\ns{1}.
   \end{align}
\end{proposition}

The cyclic cohomology of the preceding cyclic module is denoted by $^{H}HC^{*}(C,M)$.


\subsection{Cyclic cohomology of comodule rings}

 In this subsection we define the cyclic cohomology of comodule rings with coefficients in a SAYD module and the symmetry endowed by a $\times$-Hopf coalgebra.

Let $\Bc$ be a right $\times_C$-Hopf coalgebra,  $A$ be a right $\Bc$-comodule ring, and $M$ be a  right-left SAYD module over $\Bc$.
We set
\begin{equation}
\widetilde{C}^{ \mathcal{B},n}(A, M):=A^{\Co_{C}(n+1)}\Co_{\mathcal{B}}M,
\end{equation}
and define the  following operators on  $\widetilde{C}^{ \mathcal{B},n}(A, M)$.
\begin{align}
\begin{split}\label{module-ring-cocyclic}
&\delta_{i}(\td a\otimes m)= a_0\ot \cdots \ot \eta(a_i\sns{-1})\ot a_i\sns{0}  \ot \cdots\ot a_n \ot m, \quad 0\leq i \leq n,\\
& \s_i(\td a\otimes m)= a_0\ot \cdots \ot \mathfrak{m}(a_i \ot a_{i+1})\ot  \cdots\ot a_n \ot m, \qquad \quad 0\leq i \leq n,\\
&\tau_{n}(\td a\otimes m)= a_1\ot \cdots \ot a_n\ot a_0\nsb{0}\ot m\triangleleft a_0\nsb{1},
\end{split}
\end{align}
where $\td a= a_0\ot \cdots \ot a_n$. The right $\mathcal{B}$-comodule structure of  $A^{\Co_{C}(n+1)}$ is given by
\begin{equation*}
  \td a \longmapsto a_0\nsb{0}\ot \cdots \ot a_n\nsb{0}\ot a_0\nsb{1} \cdots a_n\nsb{1}.
\end{equation*}
\begin{proposition}\label{2way}
  Let $\mathcal{B}$ be a right $\times_C$-Hopf coalgebra, $M$ be a right-left SAYD module for $\mathcal{B}$ and $A$ be a right $\mathcal{B}$-comodule $C$-ring.
   Then the operators presented in \eqref{module-ring-cocyclic} define  a cocyclic module structure on $\widetilde{C}^{ \mathcal{B},n}(A, M)$.
\end{proposition}
\begin{proof}
It is straightforward  to check that the operators satisfy the commutativity relations for a cocyclic module \cite{NCG}. However,  we check that the operators are well-defined; using comodule ring conditions \eqref{mc1} and \eqref{mc2} one proves that all cofaces except possibly the last one  are well-defined. This is readily seen  for the codegeneracies. Based on the relations in the cocyclic category,  it suffices to check that  the cyclic operator is well-defined.
If $n=0$, then $\tau=\Id$ and it is obviously well-defined. The following computation  shows that $\tau(\widetilde{a} \ot_{\mathcal{B}} m)  \in  \widetilde{C}^{ \mathcal{B},n}(A, M)$.

 \begin{align*}
 &  a_1\ot \cdots \ot a_n\ot a_0\nsb{0}\ot (m\triangleleft a_0\nsb{1})\nsb{-1}\ot  (m\triangleleft a_0\nsb{1})\nsb{0}\\
 &=a_1\ot \cdots \ot a_n\ot a_0\nsb{0}\ot \nu^+(a_0\nsb{1}\ps{2}, m\ns{-1}) a_0\nsb{1}\ps{1}\\
 &\ot m\nsb{0}\triangleleft \nu^-(a_0\nsb{1}\ps{2}, m\nsb{-1})\\
 &=a_1\nsb{0}\ot \cdots \ot a_n\nsb{0}\ot a_0\nsb{0}\nsb{0}\ot \nu^+(a_0\nsb{0}\nsb{1}\ps{2}, a_0\nsb{1}\cdots\\
 &\cdots a_n\nsb{1}) a_0\nsb{0}\nsb{1}\ps{1}\ot m\triangleleft \nu^-(a_0\nsb{0}\nsb{1}\ps{2}, a_0\nsb{1}\cdots a_n\nsb{1})\\
 &=a_1\nsb{0}\ot \cdots \ot a_n\nsb{0}\ot a_0\nsb{0}\ot \nu^+(a_0\ns{1}\ps{2}, a_0\nsb{1}\ps{3}a_1\nsb{1}\cdots \\
 &\cdots a_n\nsb{1}) a_0\nsb{1}\ps{1}\ot m\triangleleft \nu^-(a_0\nsb{1}\ps{2}, a_0\nsb{1}\ps{3}a_1\nsb{1}\cdots a_n\nsb{1})\\
 &=a_1\nsb{0}\ot \cdots \ot a_n\nsb{0}\ot a_0\nsb{0}\ot a_1\nsb{1}\cdots a_n\nsb{1} a_0\nsb{1}\ps{1}\ot m\triangleleft a_0\nsb{1}\ps{2}\\
 &=a_1\nsb{0}\ot \cdots \ot a_n\nsb{0}\ot a_0\nsb{0}\nsb{0}\ot a_1\nsb{1}\cdots a_n\nsb{1} a_0\nsb{0}\nsb{1}\ot m\triangleleft a_0\nsb{1}.
  \end{align*}
  We use the AYD condition \eqref{SAYD2} in the first equality,  the fact that $\widetilde{a}\ot m\in \widetilde{C}^{ \mathcal{B},n}(A, M)$  in the second equality,  and Lemma \eqref{property2}(ii) in the fourth equality.

\end{proof}
We denote the cyclic cohomology of this cocyclic module by $\widetilde{HC}^{\mathcal{B},*}(A, M)$. It is clear that  $\mathcal{B}$ is a right $\mathcal{B}$-comodule $C$-ring by considering the comultiplication $\Delta_{\mathcal{B}}$ as the coaction and $\mu=\mathfrak{m}$. One  defines  the following map
\begin{align}\label{fi}\nonumber
 & \varphi_n:   \widetilde{C}^{ \mathcal{B},n}(\mathcal{B}, M)\longrightarrow \mathcal{B}^{\Co_{C}n}\Co_{C_{{\rm cop}}} M ,\\
  &\varphi_n(b_0\ot \cdots \ot b_n\ot m)=b_0\ot\cdots \ot b_{n-1}\ot \ve_{\mathcal{B}}(b_{n})\ot m.
\end{align}
The right $C_{{\rm cop}}$-comodule structure of $\mathcal{B}^{\Co_{C}n}$ is given by
\begin{equation}
  b_1\ot\cdots \ot b_n\longmapsto b_1\ps{1}\ot b_2\ot \cdots \ot b_n\ot \beta(b_1\ps{2}),
\end{equation}
 and the left $C_{{\rm cop}}$-comodule structure of $M$ by  $m\longmapsto \beta(m\nsb{-1})\ot m\nsb{0}$.
\begin{proposition}
  The map $\varphi_n$ defined in \eqref{fi} is a well-defined isomorphism of vector spaces.
\end{proposition}
\begin{proof} Since $b_0\ot \cdots \ot b_n\ot m\in \widetilde{C}^{ \mathcal{B},n}(\mathcal{B}, M)$, we have
\begin{equation}\label{fff}
  b_0\ps{1}\ot\cdots \ot b_n\ps{1}\ot b_0\ps{2}\cdots b_n\ps{2}\ot m= b_0\ot \cdots \ot b_n\ot m\nsb{-1}\ot m\nsb{0}.
\end{equation}
  We  prove that the map $\varphi$ is well-defined. Indeed,
\begin{align*}
  &b_0\ot \cdots \ot b_{n-1}\ve(b_n)\ot \beta(m\nsb{-1})\ot m\nsb{0}\\
  &=b_0\ps{1}\ot \cdots \ot b_{n-1}\ps{1}\ve(b_n\ps{1})\ot \beta(b_0\ps{2}\cdots  b_n\ps{2})\ot m\\
  &=b_0\ps{1}\ot \cdots \ot b_{n-1}\ps{1}\ve(b_n\ps{1})\ot \beta(b_0\ps{2})\ve(b_1\ps{2})\cdots \ve(b_n\ps{2})\ot m\\
  &=b_0\ps{1}\ot b_1\ot\cdots\ot b_{n-1}\ve(b_n)\ot \beta (b_0\ps{2})\ot m.
\end{align*}
The equation \eqref{fff} is used in the first equality, and \eqref{rightbi}(iv) and \eqref{var2} are used in the second equality. The coalgebra properties of $\mathcal{B}$ is used in the last equality.
We claim that the following map defines a two sided inverse map of $\varphi$.
\begin{align}
  &\varphi^{-1}_n(b_1\ot \cdots \ot b_n\ot m)=  b_1\ps{1}\ot\cdots \ot b_n\ps{1}\ot \\
  &\ve[\nu^{-}(b_1\ps{2}\cdots b_n\ps{2}, m\nsb{-1})] \nu^{+}(b_1\ps{2}\cdots b_n\ps{2},m\nsb{-1}) \ot m\nsb{0}.
\end{align}
 By using Lemma \eqref{property2}(iii),  the relation \eqref{var2}, and the counital property of the coaction, one  observes  $\varphi \varphi^{-1}= Id$.
We check   $\varphi^{-1} \varphi=Id$,
\begin{align*}
  &\varphi^{-1}(\varphi_n(b_0\ot \cdots \ot b_n\ot m))= \varphi^{-1}(b_0\ot\cdots \ot b_{n-1}\ve_{\mathcal{B}}(b_{n})\ot m)\\
  &=b_0\ps{1}\ot\cdots \ot b_{n-2}\ps{1}\ot b_{n-1}\ps{1}\ve(b_n)\ot \\
  &\ve[\nu^{-}(b_0\ps{2}\cdots\ot b_{n-1}\ps{2}, m\ns{-1})] \nu^{+}(b_0\ps{2}\cdots b_{n-1}\ps{2}),m\ns{-1}) \ot m\ns{0}\\
  &=b_0\ps{1}\ot\cdots \ot b_{n-2}\ps{1}\ot b_{n-1}\ps{1}\ve(b_n\ps{1})\ot \\
  &\ve[\nu^{-}(b_0\ps{2}\cdots b_{n-1}\ps{2}, b_0\ps{3}\cdots b_{n-1}\ps{3} b_n\ps{2})] \\
  &\nu^{+}(b_0\ps{2}\cdots b_{n-1}\ps{2}, b_0\ps{3}\cdots b_{n-1}\ps{3} b_n\ps{2})\ot m\\
  &=b_0\ps{1}\ot\cdots \ot b_{n-2}\ps{1}\ot b_{n-1}\ps{1}\ot \ve(b_0\ps{2})\cdots \ve(b_{n-1}\ps{2}) b_n\ot m\\
  &= b_0\ot \cdots \ot b_n\ot m.
\end{align*}
We use \eqref{fff} in the third equality. For  the fourth equality we use   $\ve(b_n\ps{1})b_n\ps{2}=b_n$, the relation \eqref{var2},  and Lemma \eqref{property2}(ii).  This  also  proves that $\varphi^{-1}$ is well-defined.
\end{proof}

Therefore we obtain the following operators by transferring   the cocyclic structure of $\widetilde{C}^{ \mathcal{B},n}(\mathcal{B}, M)$ on  $\mathcal{B}^{\Co_{C}n}\Co_{C_{{\rm cop}}} M$ .
\begin{align}\label{dual-cocyclic-module}
\begin{split}
 &\delta_{i}(\widetilde{b}\ot m)= b_1\ot \cdots \ot \eta(b_i\sns{-1})\ot b_i\sns{0}  \ot \cdots\ot b_n \ot m,\\
 &\delta_{n}(\widetilde{b}\ot m)= b_1\ps{1}\ot\cdots \ot b_n\ps{1}\ot \\
  &\ve[\nu^{-}(b_1\ps{2}\cdots b_n\ps{2}, m\nsb{-1})] \nu^{+}(b_1\ps{2}\cdots b_n\ps{2}, m\nsb{-1}) \ot m\nsb{0}.\\
 &\s_{i}(\widetilde{b}\ot m)= b_1\ot \cdots \ot b_i  b_{i+1}\ot  \cdots\ot b_n \ot m\\
 &\s_{n}(\widetilde{b}\ot m)=b_1\ot \cdots \ot b_{n-1}\ve(b_n)  \\
 &\tau_{n}(\widetilde{b}\ot m)= b_2\ps{1}\ot\cdots\ot b_n\ps{1}\ot \\
 &\ve[\nu^-(b_1\ps{2}\cdots  b_n\ps{2}, m\nsb{-1})]\nu^+(b_1\ps{2}\cdots  b_n\ps{2}, m\nsb{-1})\ot m\nsb{0}\triangleleft b_1\ps{1}
 \end{split}
\end{align}
where $\widetilde{b}= b_1\ot\cdots \ot b_n$.

\medskip

\section{Equivariant Hopf Galois coextensions}
In this section we define  equivariant Hopf Galois coextensions of $\times$-Hopf coalgebras;  that is a quadruple $(\mathcal{K}, \mathcal{B}, T, S)$ satisfying ceratin properties(  Definition \ref{equ}). We show that any equivariant Hopf Galois coextension defines a functor from the category of SAYD modules over $\mathcal{K}$ to the category of SAYD modules over $\mathcal{B}$ such that their Hopf cyclic complexes with corresponding coefficients  are isomorphic.
\subsection{Basics of equivariant Hopf Galois coextensions}

Let $\mathcal{B}$ be a right $\times_C$-Hopf coalgebra and $T$ be a right $\mathcal{B}$-module coalgebra as defined in \ref{module-coalgebra}. We set
$$I= \{ t\triangleleft b- \ve(b)t, \quad b \ot t \in\mathcal{B}\Co_C T \}.$$
One observes that  $I$ is a coideal of $T$. To see this we show  $\Delta(I)\subseteq I\ot T\ot T\ot I$. In fact for any $t\triangleleft b-\varepsilon(b)t\in I$ we have
\begin{align*}
  &\Delta(t\triangleleft b-\varepsilon(b)t)=t\ps{1}\triangleleft b\ps{1}\ot t\ps{2}\triangleleft b\ps{2}-\varepsilon(b)t\ps{1}\ot t\ps{2}\\
  &=t\ps{1}\triangleleft b\ps{1}\ot t\ps{2}\triangleleft b\ps{2}-\varepsilon(b)t\ps{1}\ot t\ps{2}-\varepsilon(b\ps{1})t\ps{1}\ot t\ps{2}\triangleleft b\ps{2}+ t\ps{1}\ot t\ps{2}\triangleleft b\\
  &=(t\ps{1}\triangleleft b\ps{1}-\varepsilon(b\ps{1})t\ps{1})\ot t\ps{2}\triangleleft b\ps{2}+ t\ps{1}\ot(t\ps{2}\triangleleft b-\varepsilon(b)t\ps{2})\\
  &\in I\ot T\ot T\ot I.
\end{align*}

Recall that given  coalgebras $C$ and $D$, a surjective coalgebra map $\pi: C\longrightarrow D$ is called a coalgebra coextension. Thus we have a coalgebra coextension $\pi: T\longrightarrow S$.

  \begin{definition}\label{equ}
    Let $C$ be a coalgebra, $\mathcal{B}$ be a right $\times_C$-Hopf coalgebra, $T$ be a right $\mathcal{B}$-module coalgebra, $S=T_{\mathcal{B}}$, $\mathcal{K}$ be a
left $\times_S$-Hopf coalgebra, and $T$ be a left $\mathcal{K}$-comodule coalgebra. Furthermore we assume that the $\Kc$ coaction is $C$-colinear. Then $T$ is called a $\mathcal{K}$-equivariant $\mathcal{B}$-Galois coextension of $S$, if the canonical map
    \begin{equation}\label{can}
      \can: T\Co_{C}\mathcal{B}\longrightarrow T\Co_{S} T, \quad t\ot b\longmapsto t\ps{1} \ot t\ps{2}\triangleleft b,
    \end{equation}
    is bijective and the right action of $\mathcal{B}$ on $T$ is $\mathcal{K}$-equivariant, \ie
    \begin{equation}\label{k-equivariant}
      (t\triangleleft b)\ns{-1} \Co_{S_{\rm cop}} (t\triangleleft b)\ns{0}= t\ns{-1}\Co_{S_{\rm cop}} t\ns{0}\triangleleft b, \quad t\ot b\in T\Co_C\mathcal{B}.
    \end{equation}
  \end{definition}
       In Definition \ref{equ}, the $S$-bicomodule structure of $T$ is given by
    \begin{equation}\label{kmk}
      t\longmapsto \pi(t\ps{1})\ot t\ps{2}, \quad t\longmapsto t\ps{1}\ot \pi(t\ps{2}),
    \end{equation}
    where the coalgebra map $\pi: T\longrightarrow S$ is the natural quotient map. We note that for any equivariant Hopf Galois coextension  we have
\begin{equation}\label{piaction}
  \pi(t\triangleleft b)=\ve(b)\pi(t), \qquad t\in T, b\in \mathcal{B}.
  \end{equation}

Using \eqref{kmk}, \eqref{piaction} and the fact that the comultiplication of $T$ is a morphism of right $C$-comodules one verifies that$\can$ defined in \eqref{can} is well-defined.
    We denote a $\mathcal{K}$-equivariant $\mathcal{B}$-Galois coextension in  Definition \ref{equ} by $^{\mathcal{K}}T(S)_{\mathcal{B}}$.  One notes that by $\mathcal{K}$-equivariant property of the right action  of $\mathcal{B}$ on $T$,$\can$ is a left $\mathcal{K}$-comodule map where the left $\mathcal{K}$-comodule structures of  $T\Co_{S} T$ and $T\Co_C \mathcal{B}$ are given by
    \begin{align}\nonumber
      &t\Co_S t'\longmapsto t\ns{-1} t'\ns{-1}\Co_{S_{\rm cop}}  t\ns{0}\Co_S t'\ns{0}, \\
      & t\Co_C b\longmapsto t\ns{-1}\Co_{S_{\rm cop}} t\ns{0}\Co_C b.
    \end{align}
    We define the left $C$-comodule structures on $T\Co_S T$  and $T\Co_C \mathcal{B}$  by
 $$t\Co_S t'\longmapsto t\sns{-1}\ot t\sns{0}\Co_S t', \quad \text{ and}\quad t\Co_C b\longmapsto \beta(b\ps{1})\ot t\Co_C b\ps{2}.$$
Similarly we define the right $C$-comodule structures on $T\Co_S T$  and $T\Co_C \mathcal{B}$  by
$$t\Co_S t'\longmapsto t\Co_S t'\sns{0}\ot t'\sns{1}, \quad \text{and }\quad  t\Co_C b\longmapsto t\Co_C b\ps{1}\ot \alpha(b\ps{2}).$$
    The map $\can$ is a right $\mathcal{B}$-module map where the right $\mathcal{B}$-module structures of $T\Co_{C}\mathcal{B}$ and $ T\Co_{S} T$ are given by
    \begin{equation}
      (t\Co_C b)\triangleleft b':=t\Co_C bb', \qquad \text {and} \quad (t\Co_S t')\triangleleft b:= t\Co_S t'\triangleleft b'.
    \end{equation}
One notes that the above action $$ T\Co_S (T\Co_C \mathcal{B})  \cong (T\Co_S T)\Co_C \mathcal{B}\longrightarrow T\Co_S T,$$ is well-defined by the fact that $T$ is a $S-C$ bicomodule; see  Definition \ref{module-coalgebra}.
 Furthermore $\can$ is a right and left $C$-comodule map via the preceding $C$-comodule structures. We leave to the reader to verify that the preceding actions and coactions are well-defined.
We denote the inverse of the Galois map \eqref{can} by the following  index notation
\begin{equation}
  \can^{-1}(t\Co_S t')=\can\ns{-}(t\Co_S t')\Co_C \can\ns{+}(t\Co_S t').
\end{equation}
If there is no confusion we write $\can^{-1}=\can\ns{-}\ot \can\ns{+}$.

We state  the properties of the maps $\can$ and $\can^{-1}$ in the  following lemma.


\begin{lemma}\label{canproperty} Let $^{\mathcal{K}}T(S)_{\mathcal{B}}$ be a $\mathcal{K}$-equivariant $\mathcal{B}$-Galois coextension. Then the following properties hold.\\
  \begin{enumerate}
\item[i)] $\can\ns{-}\ps{1}\Co_S \can\ns{-}\ps{2}\triangleleft \can\ns{+}= \Id_{T\Co_S T}, $
\item[ii)] $\can\ns{-}\left(t\ps{1} \Co_C t\ps{2}\triangleleft b\right)\Co_C \can\ns{+}\left(t\ps{1} \Co_{ C} t\ps{2}\triangleleft b\right)= t\Co_C b . $
\item[iii)]$\can\ns{-}(t\Co_S t') \triangleleft \can\ns{+}(t\Co_S t')= \ve(t) t'.$
\item[iv)]$ \ve\left(\can\ns{+}(t\Co_S t')\right)\can\ns{-}(t\Co_S t')=t\ve(t').$
\item[v)]$\can\ns{-}(t\Co_S t')\ns{-1}\Co_{S_{\rm cop}} \can\ns{-}(t\Co_S t')\ns{0}\Co_C \can\ns{+}(t\Co_S t')\\
=t\ns{-1}t'\ns{-1}\Co_{S_{\rm cop}} \can_-(t\ns{0}\ot t'\ns{0})\Co_C \can_+(t\ns{0}\ot t'\ns{0}).$
\item[vi)]$t\sns{-1}\ot \can\ns{-}(t\sns{0}\Co_S t')\Co_C \can\ns{+}(t\sns{0}\Co_S t')=\\
\beta\left([\can\ns{+}(t\Co_S t')]\ps{2}\right)\ot \can\ns{-}(t\Co_S t')\Co_C [\can\ns{+}(t\Co_S t')]\ps{1}.$
\item[vii)]$\can\ns{-}(t\Co_S t'\sns{0})\Co_C \can\ns{+}(t\Co_S t'\sns{0})\ot t'\sns{1}\\
=\can\ns{-}(t\Co_S t')\Co_C [\can\ns{+}(t\Co_S t')]\ps{1}\ot \alpha\left([\can\ns{+}(t\Co_S t')]\ps{2}\right).$
\item[viii)]$ \beta\left(\can\ns{+}(t\Co_S t')\right)\ot \can\ns{-}(t\Co_S t')= \ve(t')t\sns{-1}\ot t\sns{0}.$
\item[ix)]$\can\ns{-}(t\Co_S t')\ot \alpha\left(\can\ns{+}(t\Co_S t')\right)=\ve(t) t'\sns{0}\ot t'\sns{1}.$
\item[x)]$\left[\can\ns{-}(t\Co_S t')\right]\ps{1}\ot \left\{\left[\can\ns{-}(t\Co_S t')\right]\ps{2}\triangleleft \nu^-\left(b\ot \can\ns{+}(t\Co_S t')\right)\right\}\Co_C \\
     \nu^+\left(b\ot \can\ns{+}(t\Co_S t')\right)= t\ps{1}\ot \can\ns{-}(t\ps{2}\triangleleft b\Co_S t')\Co_C \can\ns{+}(t\ps{2}\triangleleft b\Co_S t').$
\item[xi)]$\can\ns{-}(t\triangleleft b \Co_S t')\Co_C \can\ns{+}(t\triangleleft b \Co_S t')=\\
\can\ns{-}(t\Co_S t')\triangleleft \nu^-\left(b\ot \can\ns{+}(t\Co_S t')\right)\Co_C \nu^+\left(b\ot \can\ns{+}(t\Co_S t')\right).$
\item[xii)]$\can\ns{-}(t\Co_S t'\triangleleft b)\Co_C \can\ns{+}(t\Co_S t'\triangleleft b)= \\
\can\ns{-}(t\Co_S t')\Co_C \can\ns{+}(t\Co_S t') b.$
\item[xiii)] $[\can_-(t_1\Co_S t_2)]\ps{1}\ot [\can_-(t_1\Co_S t_2)]\ps{2}\Co_C \can_+(t_1\Co_S t_2)=\\
 t_1\ps{1}\ot \can_-(t_1\ps{2}\Co_S t_2)\Co_C \can_+(t_1\ps{2}\Co_S t_2).$
\end{enumerate}
\end{lemma}
\begin{proof}
    We observe that i) and ii) are respectively  equivalent to $\can\circ \can^{-1}=Id$ and  $\can^{-1}\circ \can=Id$. To prove   iii)  we apply $\ve\ot \Id$ on  both hand sides of  i). Equation iv) is derived  by applying $\Id\ot\ve$ on  both hand sides of  i)  and then using the right $\mathcal{B}$-module coalgebra property of $T$.
  It is not difficult to see that  v) is equivalent to the left $\mathcal{K}$-comodule  property of $\can^{-1}$.
  The equations  vi) and vii) are just the left and right $C$-comodule property of  $\can^{-1}$.
    We prove   viii) by applying $\Id\ot \Id\ot \ve$ on  both hand sides of  vi). We prove ix)  by applying $\Id\ot \ve \ot \Id$ on both hand sides of vii).
 The equality
 \begin{equation}\label{can-nu}
  (Id_T\Co_S \can^{-1})\circ (\can\Co_S Id_T)= (\can\Co_C Id_{\mathcal{B}})\circ (Id_T\Co_C \nu^{-1})\circ (\can^{-1}_{13}),
 \end{equation}
 proves  x). Here we have
 \begin{align*}
& \can^{-1}_{13}: T \Co_C \mathcal{B} \Co_{S} T \longrightarrow T\Co_C \mathcal{B} \Co_{C_{{\rm cop}}} \mathcal{B},\\
&  t\ot b\ot t'\longmapsto \can\ns{-}(t\ot t')\ot b \ot \can\ns{+}(t\ot t').
  \end{align*}
 As  $\can^{-1}_{13}$ is the inverse map for
  \begin{align*}
    \can_{13}: t\ot b\ot b'\longmapsto t\ps{1}\ot b\ot t\ps{2}\triangleleft b',
  \end{align*}
   \eqref{can-nu}  is obtained by using  the bijectivity of all the involved maps in equation,
  \begin{equation}
    (Id_T\Co_S \can)\circ (\can\Co_C Id_{\mathcal{B}})=(\can\Co_S Id_T)\circ \can_{13}\circ (Id_T \Co_C \nu).
  \end{equation}

  A proof of   xi)  is obtained by applying $\ve\ot \Id\ot \Id$ on  both hand sides of x).
  We see that  xii) is equivalent to the right $\mathcal{B}$-module property of the map $\nu^{-1}$.
    The equation xiii) is equivalent to the left $T$-colinear property of the map $\can^{-1}$ where the left $T$-comodule structures of $T\Co_C \mathcal{B}$ and $T\Co_S T$ are given by
    \begin{align}
    t\Co_C b\longmapsto t\ps{1}\ot t\ps{2}\Co_C b,  \quad \text{ and} \quad t_1\Co_S t_2\longmapsto t_1\ps{1}\ot t_1\ps{2}\Co_S t_2,
    \end{align}
    respectively.

     To verify that the  $T$-comodule structure of $T\Co_C \mathcal{B}$ is well-defined, we use the $C$-comodule  compatibility of $\Bc$-module coalgebra $T$. The latter  is equivalent to
    \begin{equation}
     t\ps{1}\ot t\ps{2}\sns{0}\ot t\ps{2}\sns{1}= t\sns{0}\ps{1}\ot t\sns{0}\ps{2}\ot t\sns{1}.
\end{equation}
\end{proof}
For any $S$-bicomodule $T$, one can define,
\begin{equation}
  T^S=\left\{t\in T\;\mid \quad  t\ns{0}\varphi(t\ns{1})= t\ns{0}\varphi(t\ns{-1}), \quad \forall  \phi\in S^*\right\}
\end{equation}
We set
\begin{equation}
  T_S=\frac{T}{W_T},
\end{equation}
where
$$W_T= \left\{t\sns{0}\varphi(t\ns{1})- t\ns{0}\varphi(t\ns{-1})\;\mid\; \quad t\in T, {\varphi\in S^*}\right \}.$$
For any  coalgebra coextension  $\pi: T\twoheadrightarrow S$, we precisely obtain
\begin{equation}\label{TS}
T^S:=\left\{t\in T\;\mid\; \quad  t\ps{1} \varphi(\pi(t\ps{2}))= t\ps{2}\varphi( \pi(t\ps{1})) , \quad \forall {\varphi\in S^*}\right\},
\end{equation}
 and
\begin{equation}
W_T:=\left\{  t\ps{1} \varphi(\pi(t\ps{2}))- t\ps{2}\varphi( \pi(t\ps{1}))\;\mid\; \quad t\in T, {\varphi\in S^*}\right\}.
\end{equation}
As proved in \cite[Lemma 6.4.8, page 90]{bal2} the subspace $W_T$ is a coideal of $T$ and therefore $T_S=\frac{T}{W_T}$ is a coalgebra.
\begin{lemma}\label{SAYDaction}
Let  $^{\mathcal{K}}T(S)_{\mathcal{B}}$ be a $\mathcal{K}$-equivariant $\mathcal{B}$-Galois coextension  with the corresponding action $\triangleleft:T\Co_C \mathcal{B}\longrightarrow T$. Then, $\triangleleft$ induces a $\mathcal{B}$-action,
\begin{equation}
\blacktriangleleft:T^S\Co_C \mathcal{B}\longrightarrow T^S.
\end{equation}
\end{lemma}
\begin{proof} First we show that  $T^S$ is a right $C$-comodule. If  $t\ot b\in T\Co_C \Bc$ then we have
\begin{equation}\label{pro}
  t\sns{0}\ot t\sns{1}\ot b= t\ot b\sns{-1}\ot b\sns{0} \in T\ot C\ot \Bc.
\end{equation}
  If $t\in T^S$ then
$t\ps{1}\varphi(\pi(t\ps{2}))=t\ps{2}\varphi(\pi(t\ps{1}))$ for all $\varphi\in S^*$.
Therefore
$$t\ps{1}\varphi(\pi(t\ps{2}))\ot b\sns{-1}\ot b\sns{0}=t\ps{2}\varphi(\pi(t\ps{1}))\ot b\sns{-1}\ot b\sns{0}.$$ By \eqref{pro} we obtain

$$t\sns{0}\ps{1}\varphi(\pi(t\sns{0}\ps{2}))\ot t\sns{1}\ot b=t\sns{0}\ps{2}\varphi(\pi(t\sns{0}\ps{1}))\ot t\sns{1}\ot b.$$ This shows that
$t\sns{0}\ot t\sns{1}\ot  b\in T^S\ot C\ot \Bc$ and therefore $T^S $ is a right $C$-comodule.
 It is then enough to show that $t\triangleleft b\in T^S$ for all $t\ot b\in T^S\Co_C\Bc$.  Indeed,
\begin{align*}
&(t\triangleleft b)\ps{1}\varphi(\pi(t\triangleleft b)\ps{2}))=t\ps{1}\triangleleft b\ps{1}\varphi(\pi(t\ps{2}\triangleleft b\ps{2}))\\
&=t\ps{1}\triangleleft b\ps{1}\varphi(\pi(t\ps{2}))\ve(b\ps{2})= t\ps{1}\triangleleft b\varphi(\pi(t\ps{2}))\\
&=t\ps{2}\triangleleft b\varphi(\pi(t\ps{1}))=t\ps{2}\triangleleft b\ps{2}\varphi(\pi(t\ps{1}))\ve(b\ps{1})\\
&=t\ps{2}\triangleleft b\ps{2}\varphi(\pi(t\ps{1}\triangleleft b\ps{1}))=(t\triangleleft b)\ps{2}\varphi(\pi(t\triangleleft b)\ps{1})).
\end{align*}
We use the $\mathcal{B}$-module coalgebra property of $T$ in the first equality and $t\in T^S $ in the fourth equality.

\end{proof}

%

We  define a $S$-bicomodule structure on $T\Co_S T$ by
\begin{align}\label{coalgebras}
  t\Co_S t'\longmapsto t\Co_S t'\ps{1}\ot \pi(t'\ps{2}),\quad t\Co_S t'\longmapsto \pi(t\ps{1})\ot t\ps{2}\Co_S t'.
\end{align}
It is easy to check that the preceding coactions are well-defined. We set
\begin{equation}
  (T\Box_S T)_S= \frac{T\Co_S T}{W_{T\Co_ST}},
\end{equation}
where
\begin{equation}\label{ww}
  W_{T\Co_S T}= \langle   t\ot t'\ps{1} \varphi(\pi(t'\ps{2}))- t\ps{2}\ot t'\varphi(\pi(t\ps{1})) \mid \varphi\in S^*, t\ot t'\in T\Co_S T \rangle,
\end{equation}
\begin{lemma}\label{canline}
Let  $^{\mathcal{K}}T(S)_{\mathcal{B}}$ be a $\mathcal{K}$-equivariant $\mathcal{B}$-Galois coextension  with the canonical bijective map $\can$. Then  $\can$ induces the following  bijection;
$$\overline{\can}:T_S\Co_C \mathcal{B}\longrightarrow(T\Box_S T)_S,  \qquad \overline{t} \ot {b}\longmapsto \overline{\can(t\ot b)}.$$ Here the over line stands for a class in the quotient space.
\end{lemma}
\begin{proof} Since $W_T$ is a coideal of $T$ and the comultiplication of $T$ is a right $C$-comodule then $T_S$ is a right $C$-comodule.
To show that this map is well-defined first we show that if
$$\left(t\ps{1}\varphi(\pi(t\ps{2}))- t\ps{2}\varphi(\pi(t\ps{1}))\right)\ot b \in W_T\ot \Bc,$$ then
$$\can\left(t\ps{1}\varphi(\pi(t\ps{2}))- t\ps{2}\varphi(\pi(t\ps{1}))\right)\ot b \in W_{T\Co_S T}\ot \Bc.$$ We note that since $\can(t\ot b)= t\ps{1}\ot t\ps{2}\triangleleft b\in T\Co_S T$ then by the definition of $W_T$ in given \eqref{ww} we have;
\begin{align*}
&t\ps{1}\ot (t\ps{2}\triangleleft b)\ps{1}\varphi(\pi(t\ps{2} \triangleleft b)\ps{2})- t\ps{2}\ot t\ps{3}\triangleleft b\varphi(\pi(t\ps{1}))\\
&=t\ps{1}\ot t\ps{2}\triangleleft b\ps{1}\varphi(\pi(t\ps{3})\ve(b\ps{2}))- t\ps{2}\ot t\ps{3}\triangleleft b\varphi(\pi(t\ps{1}))\\
&=\varphi(\pi(t\ps{3}))t\ps{1}\ot t\ps{2}\triangleleft b- \varphi(\pi(t\ps{1}))t\ps{2}\ot t\ps{3}\triangleleft b   \in W_T.
\end{align*}
Now we have;
\begin{align*}
  &\can\left(t\ps{1}\varphi(\pi(t\ps{2}))\ot b- t\ps{2}\varphi(\pi(t\ps{1}))\ot b\right)\\
  &=\varphi(\can(t\ps{1}\ot b)- \varphi(\pi(t\ps{1}))\can(t\ps{2}\ot b)\\
  &= \varphi(\pi(t\ps{3}))t\ps{1}\ot t\ps{2}\triangleleft b- \varphi(\pi(t\ps{1}))t\ps{2}\ot t\ps{3}\triangleleft b.
\end{align*}
Therefore we have shown that if $t\ot b\in W_T\ot \Bc$ then $\can(t\ot b)\in W_{T\Co_ST}$. Now we show the converse implication $\can(t\ot b)\in W_{T\Co_ST} \Longrightarrow t\ot b\in W_T\ot \Bc$. One notes that the left $C$-comodule property of $\can^{-1}$ is equivalent to
     \begin{align}\label{leftcc}
     &\can_{-}(t\Co_S t')\ps{1}\ot \can_{-}(t\Co_S t')\ps{2}\ot \can_{+}(t\Co_S t')=\\ \nonumber
     &t\ps{1}\ot \can_{-}(t\ps{2}\ot t')\ot \can_{+}(t\ps{2}\ot t'),
\end{align}
 where the left $T$-comodule structures  of $T\Co_S T$ and $T\Co_C \Bc$ are given by $t\ot t'\longmapsto t\ps{1}\ot t\ps{2}\ot t'$ and $t\ot b\longmapsto t\ps{1}\ot t\ps{2}\ot b$, respectively. Also the right $S$-comodule property of the map $\can^{-1}$ is equivalent to
  \begin{align}\label{rightdd}
  &\can_{-}(t\ot t'\ps{1})\ot \can_{+}(t\ot t'\ps{1})\ot \pi(t'\ps{2})=\\ \nonumber
  &\can_{-}(t\ot t')\ps{1}\ot \can_{+}(t\ot t')\ot \pi(\can_{-}(t\ot t')\ps{2}),
\end{align}
where the right $S$-comodule structures of $T\Co_S T$ and $T\Co_C \Bc$ are given by $t\ot t'\longmapsto t\ot t'\ps{1}\ot \pi(t'\ps{2})$ and $t\ot b\longmapsto t\ps{1}\ot b\ot \pi(t\ps{2})$, respectively. Since $\can(t\ot b)\in W_{T\Co_S T}$, there are $c,c'\in T\Co_S T$ such that
$$\can(t\ot b)= c\ot c'\ps{1}\varphi(\pi(c'\ps{2}))- c\ps{2}\ot c'\varphi(\pi(c\ps{1})).$$
Therefore we have;
\begin{align*}
  &t\ot b= \can^{-1}\left(c\ot c'\ps{1}\varphi(\pi(c'\ps{2}))- c\ps{2}\ot c'\varphi(\pi(c\ps{1}))\right)\\
  &=\can_{-}(c\ot c'\ps{1})\varphi(\pi(c'\ps{2}))\ot \can_{+}(c\ot c'\ps{1})\\
  &- \can_{-}(c\ps{2}\ot c')\varphi(\pi(c\ps{1}))\ot \can_{+}(c\ps{2}\ot c')\\
  &=\can_{-}(c\ot c')\ps{1}\varphi(\pi(\can_{-}(c\ot c')\ps{2}))\ot \can_{+}(c\ot c')\\
  &- \can_{-}(c\ot c')\ps{2}\varphi(\pi(\can_{-}(c\ot c')\ps{1}))\ot \can_{+}(c\ot c')\\
  &=(\can_{-}(c\ot c')\ps{1}\varphi(\pi(\can_{-}(c\ot c')\ps{2}))\\
  &- \can_{-}(c\ot c')\ps{2}\varphi(\pi(\can_{-}(c\ot c')\ps{1}))\ot \can_{+}(c\ot c')\in W_T\ot \Bc.
\end{align*}
We use the the left $T$-comodule
 property and the right $S$-comodule property of $\can$ in the third equality.
Therefore the bijectivity of $\can$ completes the bijectivity of $\overline{\can}$.
\end{proof}

One notes that although  $T\Co_S T$ is not a coalgebra, it is proved in \cite[section 6.4, page 93-95]{bal2} that the quotient space   $(T\Co_S T)_S$  is a coalgebra by the following coproduct and counit
\begin{equation}\label{coproduct-main}
  \Delta(\overline{t\ot t'})= \overline{t\ps{1}\Co_S t'\ps{2}}\ot \overline{t\ps{2}\Co_S  t'\ps{1}}. \qquad \ve(\overline{t\ot t'})=\ve(t)\ve(t').
\end{equation}


\begin{lemma}
  The map $\kappa$ which is defined by
  \begin{equation}
  \kappa:= (\ve\ot \Id_{\mathcal{B}})\overline{\can}^{-1}:(T\Co_S T)_S\longrightarrow \mathcal{B} ,
\end{equation}
  is an anti-coalgebra map.
\end{lemma}
\begin{proof}
First we show that the map $\kappa$ is well-defined. For this we  need to show $\kappa(W_{T\Co_C T})=0$. This is equivalent to say that for any $t\ot t'\ps{1}\varphi(\pi(t'\ps{2}))- t\ps{2}\ot t'\varphi(\pi(t\ps{1})) \in W_{T\Co_C T} $ we have
    \begin{align}\label{0}
    &\kappa(t\ot t'\ps{1}\varphi(\pi(t'\ps{2}))- t\ps{2}\ot t'\varphi(\pi(t\ps{1})))\\ \nonumber
     & =\ve(\can_{-}(t\ot t'\ps{1}\varphi(\pi(t'\ps{2}))))\can_{+}(t\ot t'\ps{1})  \\ \nonumber
     &-\ve(\can_{-}(t\ps{2}\ot t'\varphi(\pi(t\ps{1})) ))\can_{+}(t\ps{2}\ot t')=0.
    \end{align}
    This can be  proved by the left $C$-comodule and right $S$-comodule properties of the map $\can^{-1}$ as follows.
  By applying $\pi \ot \ve\ot \Id$ on the left $C$-comodule property of $\can^{-1}$ given in \eqref{leftcc}  we have,
  \begin{equation}\label{1}
  \pi(\can_{-}(t\Co_S t'))\ot \can_{+}(t\Co_S t')=  \ve(\can_{-}(t\ps{2}\ot t'))\pi(t\ps{1})\ot \can_{+}(t\ps{2}\ot t').
  \end{equation}

By applying $\ve\ot \Id \ot \pi$ on the right $S$-comodule property of $\can^{-1}$ given in \eqref{rightdd} we have;
  \begin{equation}\label{2}
    \ve(\can_{-}(t\ot t'\ps{1})) \can_{+}(t\ot t'\ps{1})\varphi(\pi(t'\ps{2}))=  \can_{+}(t\ot t') \varphi(\pi(\can_{-}(t\ot t'))).
  \end{equation}
By applying \eqref{1} and \eqref{2} on the left hand side of \eqref{0} we obtain;
\begin{align*}
  &\ve(\can_{-}(t\ot t'\ps{1}\varphi(\pi(t'\ps{2}))))\can_{+}(t\ot t'\ps{1}) \\
  &- \ve(\can_{-}(t\ps{2}\ot t'\varphi(\pi(t\ps{1})) ))\can_{+}(t\ps{2}\ot t')\\
  &\can_{+}(t\ot t') \varphi(\pi(\can_{-}(t\ot t')))- \can_{+}(t\ot t') \varphi(\pi(\can_{-}(t\ot t')))=0.
\end{align*}
Therefore $\kappa$ is well-defined. To show that $\kappa$ is an anti-coalgebra map we need to show that $\D\circ\kappa=tw\circ(\kappa\ot\kappa)\circ \D$  and $\varepsilon\circ\kappa=\varepsilon$.  Since $\overline{\can}$ is bijective, this is equivalent to show $\D\circ\kappa\circ\overline{\can}=tw\circ(\kappa\ot\kappa)\circ \D\circ\overline{\can}$. The following computation proves this for all $\overline{t}\ot b\in T_S\Co_C \mathcal{B}$.
\begin{align*}
&tw\circ(\kappa\ot\kappa)\circ \D\circ\overline{\can}({t}\ot b)=tw\circ(\kappa\ot\kappa)\circ\D(\overline{t\ps{1} \ot t\ps{2}\triangleleft b})\\
&=tw\circ(\kappa\ot\kappa)\left[\overline{{t\ps{1}} \ot({t\ps{2}}\triangleleft b)\ps{2}}\ot\overline {{t\ps{2}} \ot({t\ps{2}}\triangleleft b)\ps{1}}\right]\\
&=tw\circ(\kappa\ot\kappa)\left[\overline{{t\ps{1}} \ot {t\ps{4}}\triangleleft b\ps{2}}\ot \overline{{t\ps{2}} \ot ({t\ps{3}}\triangleleft b\ps{1})}\right]\\
&=\kappa(\overline{{t\ps{2}} \ot {t\ps{3}}\triangleleft b\ps{1}})\ot \kappa(\overline{{t\ps{1}} \ot {t\ps{4}}\triangleleft b\ps{2}})\\
&=\varepsilon({t\ps{2}})b\ps{1}\ot \kappa(\overline{{t\ps{1}} \ot {t\ps{3}}\triangleleft b\ps{2}})\\
&=b\ps{1}\ot \kappa\left(\overline{{t\ps{1}} \ot \varepsilon({t\ps{2}}){t\ps{3}}\triangleleft b\ps{2}}\right)\\
&=b\ps{1}\ot \kappa(\overline{{t\ps{1}} \ot {t\ps{2}}\triangleleft b\ps{2}})=\ve({t})b\ps{1}\ot b\ps{2}=\D\circ(\varepsilon\ot \Id_{\mathcal{B}})(\overline{t}\ot b)\\
&=\D\circ(\varepsilon\ot \Id_{\mathcal{B}})\circ\overline{\can}^{-1}\overline{\can}(\overline{t}\ot b)=\D\circ\kappa\circ\overline{\can}(\overline{t}\ot b).
\end{align*}
We use Lemma \ref{canline} in the first equality and   Lemma \ref{canproperty}(ii) in the fifth and eight equalities. Moreover,
\begin{align*}
&\ve\circ \kappa(\overline{t\ot t'})=\ve\circ (\ve\ot \Id_{\mathcal{B}})\circ \overline{\can}^{-1}(\overline{t\ot t'})\\
&=\ve\circ (\ve\ot \Id_{\mathcal{B}})(\can_-(\overline{t\ot t'})\ot \can_+(\overline{t\ot t'}))\\
&=\ve(\can_-(\overline{t\ot t'}))\ve(\can_+(\overline{t\ot t'}))=\ve(t)\ve(t')=\ve(\overline{t \ot t'}).
\end{align*}
We use Lemma \ref{canproperty}(iv) in  the fourth equality.
\end{proof}

The following lemma introduce some properties of the map $\kappa$.
\begin{lemma}\label{kappaproperty}
 The map $\kappa$ satisfies the  following  properties for all ${t\ot t'}\in (T\Co_S T)_S$.
   \begin{enumerate}
\item[i)] $\kappa(\overline{t\ot t'})\ps{1}\ot \kappa(\overline{t\ot t'})\ps{2}= \kappa(\overline{t\ps{2}\ot t'\ps{1}})\ot \kappa(\overline{t\ps{1}\ot t'\ps{2}})$.
\item[ii)]$ t\ps{1}\triangleleft \kappa(\overline{t\ps{2}\ot t'})= \ve(t)t'.$
\item[iii)]$ t\ns{-1}t'\ns{-1}\ot \kappa(\overline{t\ns{0}\ot t'\ns{0}})= \ve(t\ps{1}\ns{0})t\ps{1}\ns{-1}\ot \kappa(\overline{t\ps{2}\ot t'}).$
\end{enumerate}
\end{lemma}
\begin{proof}
The relation i) is equivalent to the anti-coalgebra map property of $\kappa$.
To prove ii), we apply $ \Id\ot \ve \ot \Id$ on both hand sides of  Lemma \ref{canproperty}(xiii). We obtain
\begin{align}\label{v-v-v}\nonumber
& \can_-(\overline{t\ot t'})\ot \can_+(\overline{t\ot t'})= \ve(\can_-(\overline{t\ps{2}\ot t'}))t\ps{1}\ot \can_+(\overline{t\ps{2}\ot t'})\\
 &= t\ps{1}\ot \kappa(\overline{t\ps{2}\ot t'}).
\end{align}
By applying the right action of $\mathcal{B}$ over $T$ in the previous equation we have
\begin{align*}
 & t\ps{1}\triangleleft \kappa(\overline{t\ps{2}\ot t'})= \ve(\can_-(\overline{t\ps{2}\ot t'}))t\ps{1}\triangleleft \can_+(\overline{t\ps{2}\ot t'})\\
 &= \can_-(\overline{t\ot t'})\triangleleft \can_+(\overline{t\ot t'})=\ve(t)t'.
\end{align*}
We use  Lemma \ref{canproperty}(iii) on the last equality. For the Relation iii), we apply the relation \eqref{v-v-v} on Lemma \ref{canproperty}(v).

\end{proof}

\begin{proposition}\label{collection}
Let $^{\mathcal{K}}T(S)_{\mathcal{B}}$ be a $\mathcal{K}$-equivariant $\mathcal{B}$-Galois coextension. Then
\begin{enumerate}
\item [i)]  $T_S$ is a left $\frac{T\Co_S T}{W_{T\Co_S T}}$-comodule.
\item[ii)]  $T^S$ is a right $\frac{T\Co_S T}{W_{T\Co_S T}}$-comodule.
\item[iii)]   $T_S$ is a right $\Bc$-comodule.
\item[iv)]  $T^S$ is a left $\Bc$-comodule.
\end{enumerate}
\end{proposition}
\begin{proof}
  First we show that $T_S$ is a left $\frac{T\Co_S T}{W_{T\Co_S T}}$-comodule by the following coaction
  $$t\longmapsto \overline{t\ps{1}\Co_S t\ps{3}}\ot t\ps{2}.$$
   Since $T_S= \frac{T}{W_T}$ is a coalgebra then  $t\longmapsto \overline{t\ps{1}\ot t\ps{3}}\ot t\ps{2}\in T\ot T\ot \frac{T}{W_T}$. Also by the definition of $W_{T}$  for any $t\in \frac{T}{W_T}$ we have
   $$t\ps{1}\ot \left(\varphi(\pi(t\ps{2})) t\ps{3}- \varphi(\pi(t\ps{3})) t\ps{2}\right)\ot t\ps{4}=0\in \frac{T}{W_T},$$
   for all $\varphi\in S^*$. Since we assume that the coalgebra $T$ is defined on a field and therefore is locally projective (as mentioned in \cite[Lemma 6.4.11, page 90]{bal2}), then
   $$t\ps{1}\ot \left(\pi(t\ps{2})\ot t\ps{3}- \pi(t\ps{3})\ot t\ps{2}\right)\ot t\ps{4}=0\in \frac{T}{W_T}.$$ Therefore
   $$t\ps{1}\ot \pi(t\ps{2})\ot t\ps{4}\ot t\ps{3}=t\ps{1}\ot \pi(t\ps{3})\ot t\ps{4}\ot t\ps{2}.$$ This is equivalent to $t\ps{1}\ot t\ps{3}\ot t\ps{2}\in T\Co_S T\ot T$. Thus the coaction is well-defined.
    The following computation shows that the coaction is coassociative.
  \begin{align*}
    &(\overline{t\ps{1}\ot t\ps{3}})\ps{1}\ot (\overline{t\ps{1}\ot t\ps{3}})\ps{2}\ot t\ps{2}=t\ps{1}\ps{1}\ot t\ps{3}\ps{2}\ot t\ps{1}\ps{2}\ot t\ps{3}\ps{1}\ot t\ps{2}\\
    &t\ps{1}\ot t\ps{5}\ot t\ps{2}\ot t\ps{4}\ot t\ps{3}= t\ps{1}\ot t\ps{3}\ot t\ps{2}\ps{1}\ot t\ps{2}\ps{3}\ot t\ps{2}\ps{2}.
  \end{align*}
  To prove ii), we define the following right $\frac{T\Co_S T}{W_{T\Co_S T}}$-coaction on $T^S$;
  \begin{equation}
    t\longmapsto t\ps{2}\ot \overline{t\ps{3}\ot t\ps{1}}
  \end{equation}
  The following computation proves $t\ps{2}\ot \overline{t\ps{3}\ot t\ps{1}}\in T^S\ot T\ot T$.
  \begin{align*}
    &t\ps{3}\varphi(\pi(t\ps{2}))\ot \overline{t\ps{4}\ot t\ps{1}}=t\ps{3}\ot \overline{t\ps{4}\ot t\ps{1}\varphi(\pi(t\ps{2}))}\\
    &=t\ps{3}\ot \overline{t\ps{4}\ps{2}\varphi(\pi(t\ps{4}\ps{1})\ot \ve(t\ps{1})t\ps{2}}=t\ps{2}\varphi(\pi(t\ps{3}))\ot \overline{t\ps{4}\ot t\ps{1}}.
  \end{align*}
  We use the definition $W_{T\Co_S T}$ given in \eqref{ww}. Also for any $t\in T^S$ we have $t\ps{1}\varphi(\pi(t\ps{2}))=t\ps{2}\varphi(\pi(t\ps{1}))$ for all $\varphi \in S^*$. Since we assume the coalgebra $T$ is defined on a field (locally projective) this is equivalent to
  $t\ps{1}\ot \pi(t\ps{2})=t\ps{2}\ot \pi(t\ps{1})$. By applying
  $\Delta\ot \Delta \ot \Id$ on both hand sides of the preceding equality we obtain
  $$t\ps{2}\ot t\ps{3}\ot \pi(t\ps{4})\ot t\ps{1}= t\ps{3}\ot t\ps{4}\ot \pi(t\ps{1})\ot t\ps{2},$$ which is equivalent to $t\ps{2}\ot t\ps{3}\ot t\ps{1}\in T\ot T\Co_S T$.
  This shows that the coaction is well-defined. The following computation proves that the coaction is coassociative.
  \begin{align*}
    &t\ps{2}\ot (\overline{t\ps{3}\ot t\ps{1}})\ps{1}\ot (\overline{t\ps{3}\ot t\ps{1}})\ps{2}= t\ps{2}\ot t\ps{3}\ps{1}\ot t\ps{1}\ps{2}\ot t\ps{3}\ps{2}\ot t\ps{1}\ps{1}\\
    &t\ps{3}\ot t\ps{4}\ot t\ps{2}\ot t\ps{5}\ot t\ps{1}= t\ps{2}\ps{2}\ot t\ps{2}\ps{3}\ot t\ps{2}\ps{1}\ot t\ps{3}\ot t\ps{1}.
  \end{align*}
 For iii),  since $\kappa$ is an anti coalgebra map it turns the left coaction introduced in (i) to a well-defined coassociative right coaction of $\Bc$  on $T_S$ as follows
  \begin{equation}
    t\longmapsto t\ps{2} \ot \kappa(\overline{t\ps{1}\ot t\ps{3}}).
  \end{equation}
For the reader's convenience here we explain why the coaction is coassociative.

\begin{align*}
  &t\ps{2} \ot \kappa(\overline{t\ps{1}\ot t\ps{3}})\ps{1}\ot \kappa(\overline{t\ps{1}\ot t\ps{3}})\ps{2}\\
  &=t\ps{2} \ot \kappa(\overline{t\ps{1}\ps{2}\ot t\ps{3}\ps{1}})\ot \kappa(\overline{t\ps{1}\ps{1}\ot t\ps{3}\ps{2}})\\
  &=t\ps{3} \ot \kappa(\overline{t\ps{2}\ot t\ps{4}})\ot \kappa(\overline{t\ps{1}\ot t\ps{5}})\\
  &=t\ps{2}\ps{2} \ot \kappa(\overline{t\ps{2}\ps{1}\ot t\ps{2}}\ps{3})\ot \kappa(\overline{t\ps{1}\ot t\ps{3}}).\\
\end{align*}

   and for iv),  again since $\kappa$ is an anti coalgebra map it turns the right coaction introduced in (ii) to a well-defined coassociative left coaction of $\Bc$   on $T^S$ as follows
   \begin{equation}
     t\longmapsto \kappa(\overline{t\ps{3}\ot t\ps{1}})\ot t\ps{2}.
   \end{equation}

\end{proof}
\subsection{Equivariant Hopf Galois coextension as a functor}{\label{ss-3-1}
Let $^{\mathcal{K}}T(S)_{\mathcal{B}}$ be a $\mathcal{K}$-equivariant $\mathcal{B}$-Galois coextension, and   $M$ be a left-right SAYD  module over $\mathcal{K}$. We let $\mathcal{B}$ coact on $\widetilde{M}:= M\Co_{\mathcal{K}} T$ from  left by
\begin{equation}\label{coaction}
m\Co_{\Kc} t\longmapsto  \kappa\left( \overline{(t\ps{2}\ns{0})\ps{2} \Co_S t\ps{1}}\right)\ot t\ps{2}\ns{-1}\triangleright m\Co_{\Kc} (t\ps{2}\ns{0})\ps{1},
  \end{equation}
and let $\mathcal{B}$ act on $\widetilde{M}$ from  right by
\begin{equation}\label{action}
  (m\Co_{\Kc} t)\triangleleft b= m\Co_{\Kc} (t\triangleleft b).
\end{equation}
The right $C$-coaction of $M\Co_{\mathcal{K}} T$ is defined by
\begin{equation}
  m\Co_{\Kc}t \longmapsto m \Co_{\Kc} t\sns{0} \ot t\sns{1}.
\end{equation}
Before proving that the coaction is well-defined we explain it more.
By \eqref{v-v-v}, the coaction  defined in \eqref{coaction} reduces to
\begin{equation}\label{reduced-coaction}
 m\Co_{\Kc} t\longmapsto \can_+(\overline{t\ps{2}\ns{0}\ot t\ps{1}})\ot t\ps{2}\ns{-1}\triangleright m\ot \can_-(\overline{t\ps{2}\ns{0}\ot t\ps{1}}).
\end{equation}
Using the left $\Kc$-comodule coalgebra property of $T$, the coaction  defined in \eqref{coaction} reduces to
\begin{align}\label{third version}
  m\Co_{\Kc} t\longmapsto \kappa\left(\overline {t\ps{3}\ns{0} \Co_S t\ps{1}}\right)\ot t\ps{2}\ns{-1}t\ps{3}\ns{-1}\triangleright m\Co_{\Kc} t\ps{2}\ns{0}.
\end{align}
We notice the similarities  between the coaction \eqref{third version} and the coaction defined in Proposition \ref{collection}(iv).
In fact  $\widetilde{M}=M\Co_{\Kc} T$ is a right $(T\Co_S T)_S$-comodule by the following coaction,
\begin{equation}
  m\Co_{\Kc} t\longmapsto t\ps{2}\ns{-1}t\ps{3}\ns{-1}\triangleright m \Co_{\Kc} t\ps{2}\ns{0}\ot \overline{ t\ps{3}\ns{0}\Co_S t\ps{1}},
\end{equation}
 and the anti-coalgebra map property of $\kappa$ turns $\widetilde{M}$ to a left $\Bc$-comodule.

The following computation shows that the coaction \eqref{coaction} is well-defined on the second cotensor.
\begin{align*}
  &\kappa\left( \overline{t\ps{2}\ns{0}\ps{2} \Co_S t\ps{1}}\right)\ot \left(t\ps{2}\ns{-1}\triangleright m\right)\ns{0}\ot \left(t\ps{2}\ns{-1}\triangleright m\right)\ns{1} \ot t\ps{2}\ns{0}\ps{1}\\
  &= \kappa\left( \overline{t\ps{2}\ns{0}\ps{2} \Co_S t\ps{1}}\right)\ot \nu^+(m\ns{1}, t\ps{2}\ns{-1}\ps{1})m\ns{0}\ot t\ps{2}\ns{-1}\ps{2} \nu^-(m\ns{1}, t\ps{2}\ns{-1}\ps{1})\ot \\
  &\ot t\ps{2}\ns{0}\ps{1}\\
  &= \kappa\left( \overline{t\ns{0}\ps{2}\ns{0}\ps{2} \Co_S t\ns{0}\ps{1}}\right)\ot \nu^+(t\ns{-1}, t\ns{0}\ps{2}\ns{-1}\ps{1})m\ot \\
  &\ot t\ns{0}\ps{2}\ns{-1}\ps{2} \nu^-(t\ns{-1}, t\ns{0}\ps{2}\ns{-1}\ps{1})\ot t\ns{0}\ps{2}\ns{0}\ps{1}    \\
  &= \kappa\left( \overline{t\ns{0}\ps{2}\ps{2}\ns{0} \Co_S t\ns{0}\ps{1}}\right)\ot \nu^+(t\ns{-1}, t\ns{0}\ps{2}\ps{1}\ns{-1}\ps{1} t\ns{0}\ps{2}\ps{2}\ns{-1}\ps{1})m\ot \\
  &\ot t\ns{0}\ps{2}\ps{1}\ns{-1}\ps{2}  t\ns{0}\ps{2}\ps{2}\ns{-1}\ps{2}\nu^-(t\ns{-1}, t\ns{0}\ps{2}\ps{1}\ns{-1}\ps{1} t\ns{0}\ps{2}\ps{2}\ns{-1}\ps{1} )\ot t\ns{0}\ps{2}\ps{1}\ns{0}   \\
  &=\kappa\left( \overline{t\ps{2}\ns{0}\ps{2}\ns{0} \Co_S t\ps{1}\ns{0}}\right)\ot \nu^+(t\ps{1}\ns{-1}t\ps{2}\ns{-1}, t\ps{2}\ns{0}\ps{1}\ns{-1}\ps{1} t\ps{2}\ns{0}\ps{2}\ns{-1}\ps{1})m\ot \\
  &\ot t\ps{2}\ns{0}\ps{1}\ns{-1}\ps{2}  t\ps{2}  \ns{0}\ps{2}\ns{-1}\ps{2} \nu^-(t\ps{1}\ns{-1}t\ps{2}\ns{-1}, t\ps{2}\ns{0}\ps{1}\ns{-1}\ps{1} t\ps{2}\ns{0}\ps{2}\ns{-1}\ps{1} )\ot \\
  &\ot t\ps{2}\ns{0}\ps{1}\ns{0}   \\
  &=\kappa\left( \overline{(t\ps{2}\ns{0})\ps{2} \Co_S t\ps{1}}\right)\ot\nu^+(t\ps{2}\ns{0}\ps{1}\ns{-1}t\ps{2}\ns{-1}\ps{1}, t\ps{2}\ns{-1}\ps{2})m\ot \\
  &\ot \nu^-(t\ps{2}\ns{0}\ps{1}\ns{-1}t\ps{2}\ns{-1}\ps{1}, t\ps{2}\ns{-1}\ps{2})\ot t\ps{2}\ns{0}\ps{1}\ns{0} \\
  &=\kappa\left( \overline{(t\ps{2}\ns{0})\ps{2} \Co_S t\ps{1}}\right)\ot t\ps{2}\ns{-1}\triangleright m \ot \left((t\ps{2}\ns{0})\ps{1}\right)\ns{-1}\ot \left((t\ps{2}\ns{0})\ps{1}\right)\ns{0}.
\end{align*}
We used the left-right anti Yetter-Drinfeld condition of $M$ in the first equality. In the second equality we used the fact that $m\ot t\in M\Co_{\Kc}T$.
We used the $\Kc$-comodule coalgebra property of $T$ in \eqref{comodulecoalgebra-2} in the third and fourth equalities. The coassociativity of the coaction was used in the fifth equality. We used Lemma \ref{property} (ii) in the last equality.
The following computation shows that the coaction is well-defined on the first cotensor.
\begin{align*}
  &\kappa\left(\overline {t\ps{3}\ns{0}\ps{1}\ot \pi( t\ps{3}\ns{0}\ps{2} )\ot  t\ps{1}}\right)\ot t\ps{2}\ns{-1}t\ps{3}\ns{-1}\triangleright m\Co_{\Kc} t\ps{2}\ns{0}\\
  &=\kappa\left(\overline {t\ps{3}\ps{1}\ns{0}\ot \pi( t\ps{3}\ps{2} \ns{0})\ot  t\ps{1}}\right)\ot t\ps{2}\ns{-1}t\ps{3}\ps{1}\ns{-1}t\ps{3}\ps{2}\ns{-1}\triangleright m\Co_{\Kc} t\ps{2}\ns{0}  \\
  &=\kappa\left(\overline {t\ps{3}\ns{0}\ot \pi( t\ps{4} \ns{0})\ot  t\ps{1}}\right)\ot t\ps{2}\ns{-1}t\ps{3}\ns{-1}t\ps{4}\ns{-1}\triangleright m\Co_{\Kc} t\ps{2}\ns{0}  \\
  &=\kappa\left(\overline {t\ps{4}\ns{0}\ot \pi( t\ps{1} )\ot  t\ps{2}}\right)\ot t\ps{3}\ns{-1}t\ps{4}\ns{-1}\triangleright m\Co_{\Kc} t\ps{3}\ns{0}\\
  &=\kappa\left(\overline {t\ps{3}\ns{0}\ot \pi( t\ps{1}\ps{1} )\ot  t\ps{1}\ps{2}}\right)\ot t\ps{2}\ns{-1}t\ps{3}\ns{-1}\triangleright m\Co_{\Kc} t\ps{2}\ns{0}.\\
\end{align*}
We used the $\Kc$-comodule coalgebra property of $T$ in \eqref{comodulecoalgebra-2} in the first equality. In the third equality we used the fact that $m\ot t\in M\Co_{\Kc}T$ and $\Kc$ is a $\times_S$-Hopf coalgebra.
\begin{theorem}\label{mainSAYD}
\label{main} Let $C$  be a coalgebra,  $\mathcal{B}$ a right $\times_C$-Hopf coalgebra, $T$ be a right $\mathcal{B}$-module coalgebra, $S=T_{\mathcal{B}}$, $\mathcal{K}$ a left $\times_S$-Hopf coalgebra, $T$ be a left $\mathcal{K}$-comodule coalgebra, and $M$ be a left-right SAYD module  over $\mathcal{K}$. If $^{\mathcal{K}}T(S)_{\mathcal{B}}$ is a $\mathcal{K}$-equivariant $\mathcal{B}$-Galois coextension, then $\widetilde{M}:= M\Co_{\mathcal{K}} T$ is a right-left SAYD module over $\mathcal{B}$ by the coaction and action defined in \eqref{coaction} and \eqref{action}.
\end{theorem}
\begin{proof}   We prove that the coaction defined in \eqref{coaction}  is coassociative. We do not use the over line for the elements in the quotient to avoid cumbersome formulae.

\begin{align*}
 & \left[\kappa\left( (t\ps{2}\ns{0})\ps{2} \ot t\ps{1}\right)\right]\ps{1} \ot \left[\kappa\left( (t\ps{2}\ns{0})\ps{2} \ot t\ps{1}\right)\right]\ps{2}  \\
 &\ot t\ps{2}\ns{-1}\triangleright m\ot (t\ps{2}\ns{0})\ps{1}\\
 &=\kappa\left( (t\ps{2}\ns{0})\ps{2}\ps{2} \ot t\ps{1}\ps{1}\right) \ot \kappa\left( (t\ps{2}\ns{0})\ps{2}\ps{1} \ot t\ps{1}\ps{2}\right)\\
 &  \ot t\ps{2}\ns{-1}\triangleright m\ot (t\ps{2}\ns{0})\ps{1}\\
 &=\left[\kappa\left( (t\ps{3}\ns{0})\ps{3} \ot t\ps{1}\right)\right] \ot \left[\kappa\left( (t\ps{3}\ns{0})\ps{2} \ot t\ps{2}\right)\right] \\
 & \ot t\ps{3}\ns{-1}\triangleright m\ot (t\ps{3}\ns{0})\ps{1}\\
 &=\kappa\left( (t\ps{3}\ps{3})\ns{0} \ot t\ps{1}\right) \ot \kappa\left( (t\ps{3}\ps{2})\ns{0} \ot t\ps{2}\right) \\
 & \ot t\ps{3}\ps{1}\ns{-1}t\ps{3}\ps{2}\ns{-1}t\ps{3}\ps{3}\ns{-1}\triangleright m\ot (t\ps{3}\ps{1})\ns{0}\\
 &=\kappa\left( t\ps{5}\ns{0} \ot t\ps{1}\right) \ot \kappa\left( t\ps{4}\ns{0} \ot t\ps{2}\right) \\
  &\ot t\ps{3}\ns{-1}t\ps{4}\ns{-1}t\ps{5}\ns{-1}\triangleright m\ot t\ps{3}\ns{0}\\
 &=  \kappa\left( t\ps{5}\ns{0} \ot t\ps{1}\right)\ot \kappa\left(t\ps{4}\ns{0}\ot t\ps{2}\ns{0} \right)\\
  &\ot  t\ps{3}\ns{-1}\ps{2}t\ps{4}\ns{-1}\ps{2}t\ps{2}\ns{-1}t\ps{3}\ns{-1}\ps{1}t\ps{4}\ns{-1}\ps{1}t\ps{5}\ns{-1}m\ot t\ps{3}\ns{0}\\
    &=  \kappa\left( t\ps{4}\ns{0} \ot t\ps{1}\right)\ot \kappa\left(\left[t\ps{3}\ps{2}\right]\ns{0}\ot t\ps{2}\ns{0} \right)\\
  &\ot  t\ps{3}\ps{1}\ns{-1}\ps{2}t\ps{3}\ps{2}\ns{-1}\ps{2}t\ps{2}
  \ns{-1}t\ps{3}\ps{1}\ns{-1}\ps{1}t\ps{3}\ps{2}\ns{-1}\ps{1}t\ps{4}\ns{-1}m\\
  &\ot \left[t\ps{3}\ps{1}\right]\ns{0}\\
 \end{align*}
\begin{align*}
 &=  \kappa\left( t\ps{4}\ns{0} \ot t\ps{1}\right)\ot \kappa\left(\left[t\ps{3}\ns{0}\right]\ps{2}\ot t\ps{2}\ns{0} \right)\\
  &\ot  t\ps{3}\ns{-1}\ps{2}t\ps{2}\ns{-1}t\ps{3}\ns{-1}\ps{1}t\ps{4}\ns{-1}m\ot \left[t\ps{3}\ns{0}\right]\ps{1}\\
 &=  \kappa\left( t\ps{4}\ns{0} \ot t\ps{1}\right)\ot \kappa\left(\left[t\ps{3}\ns{0}\ns{0}\right]\ps{2}\ot t\ps{2}\ns{0} \right)\\
  &\ot  t\ps{3}\ns{0}\ns{-1}t\ps{2}\ns{-1}t\ps{3}\ns{-1}t\ps{4}\ns{-1}m\ot \left[t\ps{3}\ns{0}\ns{0}\right]\ps{1}\\
  &=  \kappa\left( t\ps{2}\ps{3}\ns{0} \ot t\ps{1}\right)\ot \kappa\left(\left[t\ps{2}\ps{2}\ns{0}\ns{0}\right]\ps{2}\ot t\ps{2}\ps{1}\ns{0} \right)\\
  &\ot  t\ps{2}\ps{2}\ns{0}\ns{-1}t\ps{2}\ps{1}\ns{-1}t\ps{2}\ps{2}\ns{-1}t\ps{2}\ps{3}\ns{-1}m\ot \left[(t\ps{2}\ps{2})\ns{0}\ns{0}\right]\ps{1}\\
 &=  \kappa\left( (t\ps{2}\ns{0})\ps{3} \ot t\ps{1}\right)\ot \kappa\left(\left[(t\ps{2}\ns{0})\ps{2}\ns{0}\right]\ps{2}\ot (t\ps{2}\ns{0})\ps{1} \right)\\
  &\ot  (t\ps{2}\ns{0})\ps{2}\ns{-1}t\ps{2}\ns{-1}m\ot \left[(t\ps{2}\ns{0})\ps{2}\ns{0}\right]\ps{1}\\
  &=  \kappa\left( (t\ps{2}\ns{0})\ps{2} \ot t\ps{1}\right)\ot \kappa\left(\left[(t\ps{2}\ns{0})\ps{1}\ps{2}\ns{0}\right]\ps{2}\ot (t\ps{2}\ns{0})\ps{1}\ps{1} \right)\\
  &\ot  (t\ps{2}\ns{0})\ps{1}\ps{2}\ns{-1}t\ps{2}\ns{-1}m\ot \left[(t\ps{2}\ns{0})\ps{1}\ps{2}\ns{0}\right]\ps{1}.
\end{align*}
We use the anti-coalgebra map property of $\kappa$ in the first equality,  the left $\mathcal{K}$-comodule coalgebra property of $T$ in the third equality, Lemma \ref{kappaproperty}[iii] and the left $\mathcal{K}$-comodule coalgebra property of $T$  in the fifth equality,  and finally   the left $\mathcal{K}$-comodule coalgebra property of $T$ in the seventh equality.

The following computation shows that the coaction defined in \eqref{coaction}  is counital.
\begin{align*}
   &\ve\left[\kappa\left( (t\ps{2}\ns{0})\ps{2} \ot t\ps{1}\right)\right] t\ps{2}\ns{-1}\triangleright m\ot (t\ps{2}\ns{0})\ps{1}\\
   &=\ve\left[ \can_-\left((t\ps{2}\ns{0})\ps{2} \ot t\ps{1} \right)\right]\ve\left[ \can_+\left((t\ps{2}\ns{0})\ps{2} \ot t\ps{1} \right)\right]t\ps{2}\ns{-1}\triangleright m\\
   &\ot (t\ps{2}\ns{0})\ps{1}\\
   &\ve[t\ps{2}\ns{0})\ps{2}] \ve(t\ps{1})t\ps{2}\ns{-1}\triangleright m\ot (t\ps{2}\ns{0})\ps{1}= \ve(t\ns{0}\ps{2})t\ns{-1}\triangleright m\ot t\ns{0}\ps{1}\\
   &=t\ns{-1}\triangleright m\ot t\ns{0}=m\ns{1}\triangleright m\ns{0}\ot t= m\ot t.
\end{align*}
We use Lemma \ref{canproperty}(iv) in the second equality, the fact that $m\ot t\in M\Co_{\mathcal{K}} T$ in the penultimate equality,  and eventually  the stability condition in the last equality.
The following computation proves the stability condition.
\begin{align*}
   & \left[t\ps{2}\ns{-1}\triangleright m\ot (t\ps{2}\ns{0})\ps{1}\right]\triangleleft \kappa\left( (t\ps{2}\ns{0})\ps{2} \ot t\ps{1}\right)\\
   &=t\ps{2}\ns{-1}\triangleright m\ot \left[(t\ps{2}\ns{0})\ps{1}\triangleleft \kappa\left( (t\ps{2}\ns{0})\ps{2} \ot t\ps{1}\right)\right]\\
   &=t\ps{2}\ns{-1}\triangleright m\ot \ve(t\ps{2}\ns{0})t\ps{1}= \ve(t\ps{2}\ns{0})\eta(\alpha(t\ps{2}\ns{-1}))= m\ot t.
\end{align*}

We use Lemma \ref{kappaproperty}(ii) in the second equality and left $\mathcal{K}$-comodule coalgebra property of $T$ in the third equality.

The following computation shows  that   the  right action $\widetilde{M}\Co_{C}\mathcal{B}\longrightarrow \widetilde{M}$, defined in \eqref{action} is well-defined.
  \begin{align*}
  & m\ns{0}\ot m\ns{1}\ot t\triangleleft b= m \ot t\ns{-1}\ot t\ns{0}\triangleleft b = m\ot (t\triangleleft b)\ns{-1}\ot  (t\triangleleft b)\ns{0}.
  \end{align*}
We use $m\ot t  \in M\Co_{\mathcal{K}}T $ in the first equality and the $\mathcal{K}$-equivariant property of the action of $\mathcal{B}$ over $T$ on the second equality. The associativity of the action defined in \eqref{action} is obvious by the associativity of the right action of $\mathcal{B}$ over  $T$.

Here we prove the AYD condition. In the following, we use the AYD condition in the first equality, the coaction defined in \eqref{coaction} in the second equality, the definition of $\kappa$ in the third equality, the equation obtained by applying $\Id\ot \ve\ot \Id$ on the left $\mathcal{B}$ module map property \eqref{b-comodule-nu} in the fourth equality, left $\mathcal{B}$-module coalgebra property of $T$ in the fifth equality,  Lemma \ref{canproperty}(xii) in the sixth equality, Lemma \ref{canproperty}(xi) in the seventh equality, left $\mathcal{B}$-module coalgebra property of $T$ in the eighth  equality and  $\mathcal{K}$-equivariant property of the right action of $\Bc$ on $T$ mentioned in  \eqref{k-equivariant} in the ninth equality.
\begin{align*}
  &\left[(m\ot t)\triangleleft b\right]\ns{-1}\ot \left[(m\ot t)\triangleleft b\right]\ns{0}\\
  &=\nu^+\left(b\ps{2}, (m\ot t)\ns{-1}\right)b\ps{1}\ot (m\ot t)\ns{0}\triangleleft \nu^-\left(b\ps{2}, (m\ot t)\ns{-1}\right)\\
  &=\nu^+\left(b\ps{2}, \kappa\left( (t\ps{2}\ns{0})\ps{2} \ot t\ps{1}\right)\right)b\ps{1} \\
  &\ot t\ps{2}\ns{-1}\triangleright m\ot \left[(t\ps{2}\ns{0})\ps{1}\triangleleft \nu^-\left(b\ps{2}, \kappa\left( (t\ps{2}\ns{0})\ps{2} \ot t\ps{1}\right)\right)\right]\\
  &=\ve\left\{ \can_-\left[ t\ps{2}\ns{0}\ps{2} \ot t\ps{1}\right]\right\}\nu^+\left(b\ps{2}, \can_+\left[ t\ps{2}\ns{0}\ps{2} \ot t\ps{1}\right]\right)b\ps{1} \\
  &\ot t\ps{2}\ns{-1}\triangleright m\ot \left[t\ps{2}\ns{0}\ps{1}\triangleleft \nu^-\left(b\ps{2}, \can_+\left[ t\ps{2}\ns{0}\ps{2} \ot t\ps{1}\right]\right)\right]\\
  &=\ve\left\{ \can_-\left[t\ps{2}\ns{0}\ps{2} \ot t\ps{1}\right]\right\} \ve\left\{ \nu^-\left(b\ps{3},\can_+\left[t\ps{2}\ns{0}\ps{2} \ot t\ps{1}\right] \right)\right\}\\
  &\nu^+\left(b\ps{3},\can_+\left[t\ps{2}\ns{0}\ps{2} \ot t\ps{1}\right] \right)b\ps{1}\ot t\ps{2}\ns{-1}\triangleright m\ot \left(t\ps{2}\ns{0}\ps{1}\triangleleft b\ps{2}\right)\\
  &=\ve\left\{ \can_-\left[t\ps{2}\ns{0}\ps{2} \ot t\ps{1}\right]\triangleleft \nu^-\left(b\ps{3},\can_+\left[t\ps{2}\ns{0}\ps{2} \ot t\ps{1}\right] \right)\right\}\\
  &\nu^+\left(b\ps{3},\can_+\left[t\ps{2}\ns{0}\ps{2} \ot t\ps{1}\right] \right)b\ps{1}\ot t\ps{2}\ns{-1}\triangleright m\ot \left(t\ps{2}\ns{0}\ps{1}\triangleleft b\ps{2}\right)\\
  \end{align*}
  \begin{align*}
  &=\ve\left\{ \can_-\left[\left(t\ps{2}\ns{0}\ps{2}\triangleleft b\ps{3}\right) \ot t\ps{1}\right]\right\}\can_+\left[\left(t\ps{2}\ns{0}\ps{2}\triangleleft b\ps{3}\right) \ot t\ps{1}\right]b\ps{1}\\
  &\ot t\ps{2}\ns{-1}\triangleright m\ot \left(t\ps{2}\ns{0}\ps{1}\triangleleft b\ps{2}\right)\\
  &=\ve\left\{ \can_-\left[\left(t\ps{2}\ns{0}\ps{2}\triangleleft b\ps{3}\right) \ot (t\ps{1}\triangleleft b\ps{1})\right]\right\}\\
  &\can_+\left[\left(t\ps{2}\ns{0}\ps{2}\triangleleft b\ps{3}\right) \ot (t\ps{1}\triangleleft b\ps{1})\right]\ot t\ps{2}\ns{-1}\triangleright m\ot \left(t\ps{2}\ns{0}\ps{1}\triangleleft b\ps{2}\right)\\
  &=\ve\left\{ \can_-\left[\left(t\ps{2}\ns{0}\triangleleft b\ps{2}\right)\ps{2} \ot (t\ps{1}\triangleleft b\ps{1})\right]\right\}\\
  &\can_+\left[\left(t\ps{2}\ns{0}\triangleleft b\ps{2}\right)\ps{2} \ot (t\ps{1}\triangleleft b\ps{1})\right]\ot t\ps{2}\ns{-1}\triangleright m\ot \left(t\ps{2}\ns{0}\triangleleft b\ps{2}\right)\ps{1}\\
  &=\ve\left\{ \can_-\left[\left(t\ps{2}\triangleleft b\ps{2}\right)\ns{0}\ps{2} \ot (t\ps{1}\triangleleft b\ps{1})\right]\right\}\\
  &\can_+\left[\left(t\ps{2}\triangleleft b\ps{2}\right)\ns{0}\ps{2} \ot (t\ps{1}\triangleleft b\ps{1})\right]\ot
  \left(t\ps{2}\triangleleft b\ps{2}\right)\ns{-1}\triangleright m\\
  &~~~~~~~~~~~~~~~~~~~~~~~~~~~~~~~~~~~~~~~~~~~~~~~~~~~~~~~~~~~~~~~~~~~~~\ot \left(t\ps{2}\triangleleft b\ps{2}\right)\ns{0}\ps{1}\\
  &=\kappa\left[\left(t\ps{2}\triangleleft b\ps{2}\right)\ns{0}\ps{2} \ot (t\ps{1}\triangleleft b\ps{1})\right]\ot \left(t\ps{2}\triangleleft b\ps{2}\right)\ns{-1}\triangleright m\\
  &~~~~~~~~~~~~~~~~~~~~~~~~~~~~~~~~~~~~~~~~~~~~~~~~~~~~~~~~~~~~~~~~~~~~~\ot \left(t\ps{2}\triangleleft b\ps{2}\right)\ns{0}\ps{1}\\
  &=\kappa\left[\left((t\triangleleft b)\ps{2}\ns{0}\right)\ps{2} \ot (t\triangleleft b)\ps{1}\right]\ot (t\triangleleft b)\ps{2}\ns{-1}\triangleright m\ot \left((t\triangleleft b)\ps{2}\ns{0}\right)\ps{1}\\
  &=\left[m\ot t\triangleleft b\right]\ns{-1}\ot \left[m\ot t\triangleleft b\right]\ns{0}.
\end{align*}

\end{proof}

\medskip

Let us denote the category of SAYD modules over a left $\times$-Hopf coalgebra $\mathcal{K}$ by $_{\mathcal{K}}SAYD^{\mathcal{K}}$. Its objects are all left-right SAYD modules over $\mathcal{K}$ and its morphisms are all $\mathcal{K}$-linear-colinear maps. Similarly one denotes by $^{\mathcal{B}}SAYD_{\mathcal{B}}$, the category of right-left SAYD modules over a right $\times$-Hopf coalgebra $\mathcal{B}$. We see that Theorem \ref{mainSAYD} amounts to a object map of a functor $\digamma$ from  $_{\mathcal{K}}SAYD^{\mathcal{K}}$ to $^{\mathcal{B}}SAYD_{\mathcal{B}}$. In fact, if $M,N\in _{\mathcal{K}}SAYD^{\mathcal{K}}$, then $\digamma(M)=\widetilde{M}$, $\digamma(N)=\widetilde{N}$. Also if $\phi: M\longrightarrow N$,  we set  $\digamma(\phi)=\phi\Co_{\mathcal{K}} Id_{T}$.
\begin{proposition}
  The assignment $\digamma: _{\mathcal{K}}SAYD^{\mathcal{K}}\longrightarrow ^{\mathcal{B}}SAYD_{\mathcal{B}}$ defines a covariant functor.
\end{proposition}
\begin{proof}
 Using Theorem \ref{mainSAYD}, the map  $\digamma$ is an object map.  It is easy to check that $\digamma(\phi)$ is a module  and comodule map.
\end{proof}
Let us recall that $$^{\mathcal{K}}C^n(T,M)= M\Co_{\mathcal{K}}T^{\Co_S (n+1)}, \quad \widetilde{C}^{\mathcal{B},n}(\mathcal{B},\widetilde{M})=\mathcal{B}^{\Co_C n}\Co_{C_{{\rm cop}}}\widetilde{M},$$ are the cocyclic modules defined in \eqref{cocyclic-module} and \eqref{dual-cocyclic-module}, respectively. We define the following map,
\begin{equation}
  \omega_n: \mathcal{B}^{\Co_C n}\Co_{C_{{\rm cop}}}\widetilde{M} \longrightarrow M\Co_{\mathcal{K}}T^{\Co_S (n+1)},
\end{equation}
given by
\begin{align*}
 &\omega_n( b_1\ot\cdots \ot b_n\ot m\ot t)= m\ot t\ps{1}\ot t\ps{2}\triangleleft b_1\ps{1}\ot t\ps{3}\triangleleft [b_1\ps{2} b_2\ps{1}]\ot\cdots\ot\\
 &t\ps{n}\triangleleft [b_1\ps{n-1}\cdots  b_{n-1}\ps{1}]\ot t\ps{n+1}\triangleleft [b_1\ps{n}\cdots  b_{n-1}\ps{2} b_n],
\end{align*}
with an inverse map
\begin{align*}
 & \omega^{-1}_n(m\ot t_0\ot \cdots\ot t_n)\\
 &=\kappa(t_0\ps{2}\ot t_1\ps{1})\ot \kappa(t_1\ps{2}\ot t_2\ps{1})\ot \cdots \ot \kappa(t_{n-1}\ps{2}\ot t_n)\ot m\ot t_0\ps{1}.
\end{align*}

\begin{theorem}\label{theorem-isomorphism}
Let $^{\mathcal{K}}T(S)_{\mathcal{B}}$ be a $\mathcal{K}$-equivariant $\mathcal{B}$-Galois coextension, and $M$ be a left-right SAYD module over $\mathcal{K}$. The map $\omega_*$ defines  an isomorphism  of cocyclic modules between $^{\mathcal{K}}C^n(T,M)$ and $\widetilde{C}^{\mathcal{B},n}(\mathcal{B},\widetilde{M})$, which are defined in \eqref{cocyclic-module} and \eqref{dual-cocyclic-module} respectively. Here $\widetilde{M}= M\Co_{\mathcal{K}} T$ is the right-left SAYD module over $\mathcal{B}$ introduced in Theorem \ref{mainSAYD}.
\end{theorem}
\begin{proof}
One notes that $\omega= \can^{\ot n}\circ tw_{\mathcal{B}^{\ot n}\ot \widetilde{M}}$, where

\begin{align*}
&\can^{\ot n}= (\Id_M\ot \Id_{T} \ot \cdots \ot \Id_T\ot \can)\circ \\
&\cdots \circ(\Id_M\ot \Id_{T}\ot \can \ot \Id_{\mathcal{B}}\ot \cdots \ot\Id_{\mathcal{B}}  )\circ (\Id_M\ot \can \ot \Id_{\mathcal{B}}\ot \cdots \ot \Id_{\mathcal{B}}).
\end{align*}

 We use Lemma \eqref{canproperty}(ii) to  easily prove $\omega^{-1}\circ \omega= \Id$. Also by anti-coalgebra map property of $\kappa$ and \eqref{v-v-v} one can show $\omega \circ \omega^{-1}= \Id.$
By a very routine and long  computation reader will check that $\omega$ is well-defined and  commutes with faces and degeneracies. Here we only show  that $\omega$ commutes with the cyclic maps.

\begin{align*}
 & \mathfrak{t}_n \omega_n(b_1\ot\cdots \ot b_n\ot m\ot t)\\
 &=\mathfrak{t}_{n} \{ m\ot t\ps{1}\ot t\ps{2}\triangleleft b_1\ps{1}\ot t\ps{3}\triangleleft [b_1\ps{2} b_2\ps{1}]\ot\cdots\\
 &~~~~~~~~~~~~~~~~~\dots \ot t\ps{n}\triangleleft [b_1\ps{n-1}\cdots  b_{n-1}\ps{1}]\ot t\ps{n+1}\triangleleft [b_1\ps{n}\cdots  b_{n-1}\ps{2} b_n]\}\\
 &=\left[t\ps{2}\triangleleft b_1\ps{1}\right]\ns{-1}\cdots  \left[t\ps{n}\triangleleft [b_1\ps{n-1}\cdots  b_{n-1}\ps{1}]\right]\ns{-1}\\
 &~~~~~~~~~ \left[t\ps{n+1}\triangleleft [b_1\ps{n}\cdots  b_{n-1}\ps{2} b_n]\right]\ns{-1}\triangleright m\\
 &~~~~~~~~~~~~~~\ot \left[t\ps{2}\triangleleft b_1\ps{1}\right]\ns{0}\ot \cdots  \ot \left[t\ps{n}\triangleleft [b_1\ps{n-1}\cdots  b_{n-1}\ps{1}]\right]\ns{0}\\
 & ~~~~~~~~~~~~~~~~~~~~~~~~~~~~~~~~~~~~~~~~~\ot \left[t\ps{n+1}\triangleleft [b_1\ps{n}\cdots  b_{n-1}\ps{2} b_n]\right]\ns{0}\ot t\ps{1}\\
 &=t\ps{2}\ns{-1}\cdots  t\ps{n}\ns{-1} t\ps{n+1}\ns{-1}\triangleright m\ot \left[t\ps{2}\ns{0}\triangleleft b_1\ps{1}\right]\ot \cdots \\
 &\dots \ot \left[t\ps{n}\ns{0}\triangleleft [b_1\ps{n-1}\cdots  b_{n-1}\ps{1}]\right] \ot \left[t\ps{n+1}\ns{0}\triangleleft [b_1\ps{n}\cdots  b_{n-1}\ps{2} b_n]\right]\ot t\ps{1}\\
 &=t\ps{2}\ns{-1}\triangleright m \ot \left[(t\ps{2}\ns{0})\ps{1}\triangleleft b_1\ps{1}\right]\ot \left[(t\ps{2}\ns{0})\ps{2}\triangleleft b_1\ps{2} b_2\ps{1}\right]\\
 &\ot\cdots \ot \left[(t\ps{2}\ns{0})\ps{n}\triangleleft b_1\ps{n}b_2\ps{n-1}\cdots b_{n-1}\ps{2}b_n\right]\ot t\ps{1} \\
 &=t\ps{2}\ns{-1}\triangleright m \ot \left[(t\ps{2}\ns{0})\ps{1}\triangleleft b_1\ps{1}\right]\ot \left[(t\ps{2}\ns{0})\ps{2}\triangleleft b_1\ps{2} b_2\ps{1}\right] \ot \cdots\\
 &~~~~~~~~~~~~~~~\dots \ot \left[(t\ps{2}\ns{0})\ps{n}\triangleleft b_1\ps{n}b_2\ps{n-1}\cdots b_{n-1}\ps{2}b_n\right]\\
 &~~~~~~~~~~~~~~~~~~~~~~~~~~~~~~~~~~\ot (t\ps{2}\ns{0})\ps{n+1}\triangleleft \kappa\left( (t\ps{2}\ns{0})\ps{n+2} \ot t\ps{1}\right) \\
 &=t\ps{2}\ns{-1}\triangleright m \ot \left[(t\ps{2}\ns{0})\ps{1}\triangleleft b_1\ps{1}\right]\ot \left[(t\ps{2}\ns{0})\ps{2}\triangleleft b_1\ps{2} b_2\ps{1}\right]\\
 &\ot\cdots \ot \left[(t\ps{2}\ns{0})\ps{n}\triangleleft b_1\ps{n}b_2\ps{n-1}\cdots b_n\ps{1}\right]\\
 &\ot (t\ps{2}\ns{0})\ps{n+1}\triangleleft \ve\left(b_1\ps{n+1} b_2\ps{n}\cdots b_n\ps{2}\right)\kappa\left( (t\ps{2}\ns{0})\ps{n+2} \ot t\ps{1}\right) \\
 &=t\ps{2}\ns{-1}\triangleright m \ot \left[(t\ps{2}\ns{0})\ps{1}\triangleleft b_1\ps{1}\right]\ot \left(\left[(t\ps{2}\ns{0})\ps{2}\triangleleft b_1\ps{2}\right]\triangleleft b_2\ps{1}\right)\ot\cdots\\
 &\cdots \ot \left(\left[(t\ps{2}\ns{0})\ps{n}\triangleleft b_1\ps{n}\right]\triangleleft[b_2\ps{n-1}\cdots b_n\ps{1}]\right)\\
 &\ot (t\ps{2}\ns{0})\ps{n+1}\triangleleft \nu^-\left(b_1\ps{n+1} b_2\ps{n}\cdots b_n\ps{2},\kappa\left( (t\ps{2}\ns{0})\ps{n+2} \ot t\ps{1}\right)\right) \\
 &~~~~~~~~~~~~~~~~~~~~~~~~~~~~~~~~~~~~\nu^+\left(b_1\ps{n+1} b_2\ps{n}\cdots b_n\ps{2},\kappa\left( (t\ps{2}\ns{0})\ps{n+2} \ot t\ps{1}\right)\right)\\
  &=t\ps{2}\ns{-1}\triangleright m \ot \left[(t\ps{2}\ns{0})\ps{1}\triangleleft b_1\ps{1}\right]\ot \left(\left[(t\ps{2}\ns{0})\ps{2}\triangleleft b_1\ps{2}\right]\triangleleft b_2\ps{1}\right)\ot\cdots \\
 &\cdots\ot \left(\left[(t\ps{2}\ns{0})\ps{n}\triangleleft b_1\ps{n}\right]\triangleleft[b_2\ps{n-1}\cdots b_n\ps{1}]\right)\\
 &~~~~~~~~~~~~\ot \left((t\ps{2}\ns{0})\ps{n+1}\triangleleft \left(b_1\ps{n+1} b_2\ps{n}\cdots b_n\ps{2}\right)\ps{1}\right.\\
 &~~~~~~~~~~~~~~~~~~\left.\nu^+\left(\left(b_1\ps{n+1} b_2\ps{n}\cdots b_n\ps{2}\right)\ps{2},  \kappa\left( (t\ps{2}\ns{0})\ps{n+2} \ot t\ps{1}\right)\right)\right)\\
 &~~~~~~~~~~~~~~~~~~~~~~~~~~~~~~\ve\left\{\nu^-\left(b_1\ps{n+2}\cdots  b_n\ps{3},  \kappa\left( (t\ps{2}\ns{0})\ps{n+2} \ot t\ps{1}\right)\right)\right\}\\
 \end{align*}

 \begin{align*}
 &=t\ps{2}\ns{-1}\triangleright m \ot \left[(t\ps{2}\ns{0})\ps{1}\triangleleft b_1\ps{1}\right]\ot \left(\left[(t\ps{2}\ns{0})\ps{2}\triangleleft b_1\ps{2}\right]\triangleleft b_2\ps{1}\right)\ot\cdots\\
 &\cdots \ot \left(\left[(t\ps{2}\ns{0})\ps{n}\triangleleft b_1\ps{n}\right]\triangleleft[b_2\ps{n-1}\cdots b_n\ps{1}]\right)\\
 &~~~~~~~~~\ot \left(\left[(t\ps{2}\ns{0})\ps{n+1}\triangleleft b_1\ps{n+1}\right]\triangleleft [b_2\ps{n}\cdots b_n\ps{2}\right.\\
 &~~~~~~~~~~~~~~~~\left.\nu^+\left(b_1\ps{n+2}\cdots  b_n\ps{3},  \kappa\left( (t\ps{2}\ns{0})\ps{n+2} \ot t\ps{1}\right)\right)]\right)\\
 &~~~~~~~~~~~~~~~~~~~~~~\ve\left\{\nu^-\left(b_1\ps{n+2}\cdots  b_n\ps{3},  \kappa\left( (t\ps{2}\ns{0})\ps{n+2} \ot t\ps{1}\right)\right)\right\}\\
 &=t\ps{2}\ns{-1}\triangleright m \ot \left[(t\ps{2}\ns{0})\ps{1}\ps{1}\triangleleft b_1\ps{1}\ps{1}\right]\ot \left(\left[(t\ps{2}\ns{0})\ps{1}\ps{2}\triangleleft b_1\ps{1}\ps{2}\right]\triangleleft b_2\ps{1}\ps{1}\right)\ot \cdots\\
 &\cdots \ot \left(\left[(t\ps{2}\ns{0})\ps{1}\ps{n}\triangleleft b_1\ps{1}\ps{n}\right]\triangleleft[b_2\ps{1}\ps{n-1}\cdots b_n\ps{1}\ps{1}]\right)\\
 &\ot \left(\left[(t\ps{2}\ns{0})\ps{1}\ps{n+1}\triangleleft b_1\ps{1}\ps{n+1}\right]\triangleleft [b_2\ps{1}\ps{n}\cdots b_n\ps{1}\ps{2}\right.\\
 &~~~~~~~~~\left.\nu^+\left(b_1\ps{2}\cdots  b_n\ps{2},  \kappa\left( (t\ps{2}\ns{0})\ps{2} \ot t\ps{1}\right)\right)]\right)\\
 &~~~~~~~~~~~~~~~~\ve\left\{\nu^-\left(b_1\ps{2}\cdots  b_n\ps{2},  \kappa\left( (t\ps{2}\ns{0})\ps{2} \ot t\ps{1}\right)\right)\right\}\\
 &=t\ps{2}\ns{-1}\triangleright m \ot \left[(t\ps{2}\ns{0})\ps{1}\triangleleft b_1\ps{1}\right]\ps{1}\ot \left(\left[(t\ps{2}\ns{0})\ps{1}\triangleleft b_1\ps{1}\right]\ps{2}\triangleleft b_2\ps{1}\ps{1}\right)\\
 &\ot\cdots \ot \left(\left[(t\ps{2}\ns{0})\ps{1}\triangleleft b_1\ps{1}\right]\ps{n}\triangleleft[b_2\ps{1}\ps{n-1}\cdots b_n\ps{1}\ps{1}]\right)\\
 &\ot \left(\left[(t\ps{2}\ns{0})\ps{1}\triangleleft b_1\ps{1}\right]\ps{n+1}\triangleleft \left[b_2\ps{1}\ps{n}\cdots b_n\ps{1}\ps{2}\right.\right.\\
 &~~~~~~~~~~~~~~~~~~~~~~~~\left.\left.\nu^+\left(b_1\ps{2}\cdots  b_n\ps{2},  \kappa\left( (t\ps{2}\ns{0})\ps{2} \ot t\ps{1}\right)\right)\right]\right)\\
 &\ve\left\{\nu^-\left(b_1\ps{2}\cdots  b_n\ps{2},  \kappa\left( (t\ps{2}\ns{0})\ps{2} \ot t\ps{1}\right)\right)\right\}\\
 &=\omega_n\{ b_2\ps{1}\ot\cdots\ot b_n\ps{1}\ot \ve[\nu^-\left(b_1\ps{2}\cdots  b_n\ps{2}, \kappa\left( (t\ps{2}\ns{0})\ps{2} \ot t\ps{1}\right)\right)]\\
 &\nu^+\left(b_1\ps{2}\cdots  b_n\ps{2},  \kappa\left( (t\ps{2}\ns{0})\ps{2} \ot t\ps{1}\right)\right)\ot  t\ps{2}\ns{-1}\triangleright m\ot (t\ps{2}\ns{0})\ps{1}\triangleleft b_1\ps{1}\}\\
 &=\omega_n\{ b_2\ps{1}\ot\cdots\ot b_n\ps{1}\ot \ve[\nu^-(b_1\ps{2}\cdots  b_n\ps{2}, (m\ot t)\ns{-1})]\\
 &\nu^+(b_1\ps{2}\cdots  b_n\ps{2}, (m\ot t)\ns{-1})\ot (m\ot t)\ns{0}\triangleleft b_1\ps{1}\}\\
 &=\omega_n \tau_n (b_1\ot\cdots \ot b_n\ot m\ot t).
\end{align*}

We use the $\mathcal{K}$-equivariant property of the action of $\Bc$ on $T$ mentioned in \eqref{k-equivariant} in the third equality,  the left $\mathcal{K}$-comodule coalgebra property of $T$ in the fourth equality, Lemma \ref{kappaproperty}(ii) in the fifth equality, the coalgebra property of $\mathcal{B}$ in the sixth equality, Lemma \ref{property2}(v) in the seventh equality,  the relation obtained by applying $\Id\ot \ve\ot \Id$ on the left $\mathcal{B}$-comodule property \eqref{b-comodule-nu} in the eighth equality, the relation \eqref{brzmil23}  in the ninth equality, the coassociativity of the comultiplication $\mathcal{B}$ in the tenth equality and  the right $\mathcal{B}$-comodule coalgebra property of $T$ the eleventh  equality.

\end{proof}


\begin{thebibliography}{9}




\bibitem[B\'al]{bal} I. B\'alint,
\emph{Scalar extension of bicoalgebroids.}
Appl. Categ. Structures {\bf 16} (2008), no. 1-2, 29–55.

\bibitem[B\'al-Thesis]{bal2}I. B\'alint, \emph{Quantum groupoid symmetry, with an application to low dimensional algebraic quantum field theory}, PhD thesis, Lor\'and E\"otv\"os University, (2008). 154 pages.

\bibitem[BO]{b1}
G. B\"{o}hm, \emph{Hopf Algebroids}, Handbook of Algebra Vol \textbf{6}, edited by M. Hazewinkel, Elsevier( 2009),  pp. 173--236.


\bibitem[BS]{bs2}
G. B\"{o}hm, D. Stefan, \emph{(co)cyclic (co)homology of bialgebroids: An approach via (co)monads}, Commun. Math. Phys\textbf{ 282} (2008),  pp. 239--286.

\bibitem[BM]{bm}
T. Brzezi´nski, G. Militaru, \emph{Bialgebroids, $\times_A$-bialgebras and duality,} J. Algebra \textbf{251}
(2002) 279-294


\bibitem[BH]{bh}
T. Brzezinski, P. Hajac, \emph{Coalgebra extensions and algebra coextensions of Galois type}. Commun. Algebra \textbf{27(3)}:1347-1368, 1999.



\bibitem[Bur]{bur} D. Burghelea. \emph{The cyclic homology of the group rings.} Comment. Math. Helv. \textbf{60} (1985), no. 3, 354–365.


\bibitem[C-Book]{NCG} A. Connes, {\bf Noncommutative geometry},  Academic Press, 1994.

\bibitem[CM98]{ConMos:HopfCyc} Connes and H. Moscovici, \emph{Hopf algebras,
  cyclic cohomology and the transverse index theorem},
  Comm. Math. Phys. \textbf{198} (1998), pp. 199--246.

\bibitem[CM01]{ConMos:DiffCyc} Connes and H. Moscovici, \emph{Differential
    cyclic cohomology and Hopf algebraic structures in transverse geometry},
    in: Essays on geometry and related topics. Vol 1-2,
    Monogr. Enseign. Math. \textbf{ 38}, Enseignement Math, Geneva (2001), pp. 217--255.

    \bibitem[GNT]{lefthopf} J. A. Green, W. D. Nichols and E. J. Taft, \emph{Left Hopf algebras.}
J. Algebra, \textbf{65 (2)} (1980), pp. 399 -- 411.

\bibitem[HKRS1]{HaKhRaSo1} P.M. Hajac, M. Khalkhali, B. Rangipour and
  Y. Sommerh\"auser,\emph{ Stable anti-Yetter-Drinfeld modules},
  C. R. Math. Acad. Sci. Paris  \textbf{338}  (2004), pp. 587--590.

\bibitem[HKRS2]{HaKhRaSo2} P.M. Hajac, M. Khalkhali, B. Rangipour and
  Y. Sommerh\"auser, \emph{Hopf-cyclic homology and cohomology with
  coefficients,}
  C. R. Math. Acad. Sci. Paris  \textbf{338}  (2004), pp. 667--672.

\bibitem[HR]{hr} M. Hassanzadeh, B. Rangipour, \emph{Equivariant Hopf Galois extensions and Hopf cyclic cohomology},
To appear in Journal of noncommutative Geometry,  arXiv:1010.5818.


\bibitem[J-S]{JaSt:CycHom} P. Jara and D. Stefan, \emph{Hopf-cyclic
    homology and relative cyclic homology of Hopf-Galois extensions},
    Proc. London Math. Soc. (3) \textbf{93} (2006), pp. 138--174.




\bibitem[KR02]{bm1}
M. Khalkhali and B. Rangipour, \emph{A new cyclic module for Hopf algebras},  K-Theory  27  (2002), pp. 111--131.

\bibitem[KK]{kk} N. Kowalzig, U. Kraehmer, \emph{Cyclic structures in algebraic (co)homology theories}, Homology, Homotopy and Applications, Vol. 13 (2011), No. 1, pp.297--318.








\bibitem[Sch]{sch} P. Schauenburg, {\em Duals and doubles of
  quantum groupoids ($\times_R$-Hopf algebras)}, in: N. Andruskiewitsch,
  W.R. Ferrer-Santos and H.-J. Schneider (eds.) AMS
  Contemp. Math. \textbf{267}, AMS Providence (2000),  pp. 273--293.


\bibitem[Schn]{schneider} H-J.  Schneider, {\em Some remarks on exact sequences of quantum groups.}
Comm. Algebra {\bf 21} (1993), no. 9, 3337–3357.

\end{thebibliography}
  \end{document}